\documentclass[12pt]{amsart}

\usepackage{amsmath,amsthm}
\usepackage{caption}
\usepackage{cite}
\usepackage{enumitem}
\usepackage{color}
\usepackage{url}
\usepackage[usenames,dvipsnames]{xcolor}
\usepackage{graphicx}
\usepackage{xspace}

\usepackage[curve]{xypic}

\usepackage{mathrsfs}

\newcommand{\zarcon}{{\mathrel{{+}\!\!\!\!\circ}}}

\usepackage{hyperref}
\usepackage{color}
\definecolor{darkred}{rgb}{0.4,0,0}
\definecolor{darkgreen}{rgb}{0,0.5,0}
\definecolor{darkblue}{rgb}{0,0,0.4}
\hypersetup{
    pdftitle={Algebraic hull of KZ cocycle},		% title
    pdfauthor={AE,SF,AW},	% author
    pdfsubject={math},		% subject of the document
    pdfkeywords={Kontsevich--Zorich cocycle, Teichmuller curves},	% list of keywords
    pdfnewwindow=true,		% links in new window
    colorlinks=true,		% false: boxed links; true: colored links
    linkcolor=darkblue,		% color of internal links (change box color with linkbordercolor)
    citecolor=darkred,		% color of links to bibliography
    filecolor=darkblue,		% color of file links
    urlcolor=darkblue,		% color of external links
    pdfborder={0 0 0},
    breaklinks=true
}

\newcommand{\Sp}{\operatorname{Sp}}
\newcommand{\GL}{\operatorname{GL}}
\newcommand{\SL}{\operatorname{SL}}
\newcommand{\Aut}{\operatorname{Aut}}
\newcommand{\Hom}{\operatorname{Hom}}
\newcommand{\Stab}{\operatorname{Stab}}

\newcommand{\boldx}{\mathbf{x}}
\newcommand{\boldy}{\mathbf{y}}
\newcommand{\Gr}{\operatorname{Gr}}
\usepackage{amssymb}
\newcommand{\onto}{\twoheadrightarrow}

\makeatletter
\renewcommand{\paragraph}{%
\@startsection {paragraph}{4}
{\z@} \z@ {-\fontdimen 2\font }\bfseries
}
\makeatother

\makeatletter % makes "@" a normal character so that it can be used in macros
\def\@cite#1#2{{\m@th\upshape\bfseries%
[{#1\if@tempswa{\m@th\upshape\mdseries, #2}\fi}]}}
\makeatother % returns "@" to its normal state

\numberwithin{equation}{section}

\theoremstyle{plain}
\newtheorem{thm}[equation]{Theorem}
\newtheorem{cor}[equation]{Corollary}
\newtheorem{prop}[equation]{Proposition}
\newtheorem{lem}[equation]{Lemma}

\theoremstyle{definition}
\newtheorem{definition}[equation]{Definition}

\theoremstyle{remark}
\newtheorem{remark}[equation]{Remark}

\renewcommand{\bold}[1]{\medskip \noindent {\bf #1 }\nopagebreak}

\newcommand{\nc}{\newcommand}
\newcommand{\rnc}{\renewcommand}
\newcommand{\ann}[1]{}
\nc\bA{\mathbb{A}}
\nc\bB{\mathbb{B}}
\nc\bC{\mathbb{C}}
\nc\bD{\mathbb{D}}
\nc\bE{\mathbb{E}}
\nc\bF{\mathbb{F}}
\nc\bG{\mathbb{G}}
\nc\bH{\mathbb{H}}
\nc\bI{\mathbb{I}}
\nc{\bJ}{\mathbb{J}} 
\nc\bK{\mathbb{K}}
\nc\bL{\mathbb{L}}
\nc\bM{\mathbb{M}}
\nc\bN{\mathbb{N}}
\nc\bO{\mathbb{O}}
\nc\bP{\mathbb{P}}
\nc\bQ{\mathbb{Q}}
\nc\bR{\mathbb{R}}
\nc\bS{\mathbb{S}}
\nc\bT{\mathbb{T}}
\nc\bU{\mathbb{U}}
\nc\bV{\mathbb{V}}
\nc\bW{\mathbb{W}}
\nc\bY{\mathbb{Y}}
\nc\bX{\mathbb{X}}
\nc\bZ{\mathbb{Z}}
\nc\cA{\mathcal{A}}
\nc\cB{\mathcal{B}}
\nc\cC{\mathcal{C}}
\rnc\cD{\mathcal{D}}
\nc\cE{\mathcal{E}}
\nc\cF{\mathcal{F}}
\nc\cG{\mathcal{G}}
\rnc\cH{\mathcal{H}}
\nc\cI{\mathcal{I}}
\nc{\cJ}{\mathcal{J}} 
\nc\cK{\mathcal{K}}
\rnc\cL{\mathcal{L}}
\nc\cM{\mathcal{M}}
\nc\cN{\mathcal{N}}
\nc\cO{\mathcal{O}}
\nc\cP{\mathcal{P}}
\nc\cQ{\mathcal{Q}}
\rnc\cR{\mathcal{R}}
\nc\cS{\mathcal{S}}
\nc\cT{\mathcal{T}}
\nc\cU{\mathcal{U}}
\nc\cV{\mathcal{V}}
\nc\cW{\mathcal{W}}
\nc\cY{\mathcal{Y}}
\nc\cX{\mathcal{X}}
\nc\cZ{\mathcal{Z}}

\nc\oGamma{\overline{\Gamma}}

\nc{\dmo}{\DeclareMathOperator}
\rnc{\Re}{\operatorname{Re}}
\rnc{\Im}{\operatorname{Im}}

\nc{\spn}{\operatorname{span}}
\dmo{\rank}{rank}
\dmo{\End}{End}

\dmo{\Jac}{Jac}
\dmo{\Id}{Id}
\dmo{\Ann}{Ann}
\dmo{\Area}{Area}
\dmo{\Isom}{Isom}

\nc{\HT}{Hodge-Teichm\"uller\xspace}
\nc{\Teichmuller}{Teichm\"uller\xspace}

\title{The algebraic hull of the Kontsevich--Zorich cocycle}
\author[Eskin]{Alex~Eskin}
\author[Filip]{Simion~Filip}
\author[Wright]{Alex~Wright}
\begin{document}

\begin{abstract}
We compute the algebraic hull of the Kontsevich--Zorich cocycle over any $ \GL^+_2(\bR) $ invariant subvariety of the Hodge bundle, and derive from this finiteness results on such subvarieties. 
\end{abstract}

\maketitle
%\tableofcontents
% removes page number from first page
%\thispagestyle{empty}

%%%%%%%%%%%%%%%%%%% 
% TABLE OF CONTENTS
%%%%%%%%%%%%%%%%%%%
% allows subsections (depth 1) to be displayed in table of contents
%\setcounter{tocdepth}{1} 
%\tableofcontents
%\vfill
%\newpage

%Last updated \today.

%\listoffixmes

%%%%%%%%%%%%%%%%%%%%%%%%%%%%%%%%%%%%%%%%%%%
%%%%%%%%%%%%%%%%%%%%%%%%%%%%%%%%%%%%%%%%%%%
\section{Introduction}
%%%%%%%%%%%%%%%%%%%%%%%%%%%%%%%%%%%%%%%%%%%
%%%%%%%%%%%%%%%%%%%%%%%%%%%%%%%%%%%%%%%%%%%

The space of Riemann surfaces equipped with a holomorphic $ 1 $-form carries a natural action of $ \GL^+_2(\bR) $.
The group of diagonal matrices corresponds to the \Teichmuller geodesic flow and can be viewed as a renormalization process for certain flows on surfaces.
This renormalization process has applications to a class of dynamical systems including interval exchange transformations and flows on surfaces.
For an introduction and survey of these topics, see Forni--Matheus \cite{Forni_Matheus_survey}, Masur--Tabachnikov \cite{Masur_Tab}, Wright \cite{Wsurvey}, and Zorich \cite{Zorich_survey}.

Topological and measure rigidity results for the $ \GL_2^+(\bR) $-action due to McMullen \cite{McMullen_classification} and  Eskin, Mirzakhani, and Mohammadi \cite{EM,EMM} show many similarities with locally homogeneous spaces and Ratner's Theorems.
In particular, if $(X,\omega)$ is a Riemann surface with holomorphic $1$-form, the  closure in the stratum of Abelian differentials of the orbit $ \overline{\GL_2^+(\bR)\cdot (X,\omega)} $  is an immersed orbifold given in certain  local coordinates by linear equations. Such immersed sub-orbifolds, usually called ``affine invariant submanifolds'',  are subvarieties of strata of Abelian differentials \cite{sfilip_algebraicity}.

Our main results analyze the Kontsevich--Zorich cocycle, which encodes  parallel translation of cohomology classes along $\GL_2^+(\bR)$ orbits. When translation surfaces are described by polygons in the plane, the $ \GL_2^+(\bR) $-action distorts the polygons linearly.
The Kontsevich--Zorich cocycle  encodes the procedure of cutting and regluing the polygons to a less distorted shape, and hence carries the mysterious part of the dynamics of the $ \GL_2^+(\bR)$-action.
It has been studied extensively, see e.g. \cite{Forni_deviations, FMZ_lyap}.

The algebraic hull of a cocycle is, informally speaking, the smallest algebraic group into which the cocycle can be conjugated \cite{Zimmer_book}.
We analyze the algebraic hull of the Kontsevich--Zorich cocycle over an arbitrary affine\ann{\textbf{Referee:} ``an arbitrary an affine"\\\textbf{Authors:} Deleted extra word.} invariant submanifold $ \cM$.
Theorem~\ref{T:AlgHull} computes it in terms of monodromy, and further results on monodromy determine precisely some parts of the algebraic hull.
Theorem~\ref{T:equidistribution} shows that for any sequence of affine manifolds $ \cM_i \subset \cM$ that equidistribute inside $ \cM $, the algebraic hulls of $ \cM $ and $ \cM_i $ eventually agree, up to finite index and compact factors.

The main applications of these results is to finiteness of orbit closures.
Theorem~\ref{T:Fin} implies that in any genus, there are only finitely many  orbit closures, with only two kinds of exceptions.
First, there are always the square-tiled surfaces, which generate \Teichmuller curves.
Second, there could be finitely many families of orbit closures of a very special kind, and such families are themselves contained in finitely many higher-dimensional orbit closures. Theorem~\ref{T:Fin}  is one of many results in mathematics stating that when infinitely many ``special" subvarieties of a given dimension exist it is because they are contained in a larger dimensional special  subvariety, whose existence implies the existence of the smaller special subvarieties. Compare to the Andr{\'e}--Oort and Zilber--Pink Conjectures. (For us the ``special" subvarieties are the $\GL^+_2(\bR)$ invariant ones.)

Methods related to those in the current paper were used by Matheus--Wright \cite{MW} and Lanneau--Nguyen--Wright \cite{LNW} to prove finiteness results for Teichm\"uller curves. Very different methods, which unlike ours are in principle effective, have been used by McMullen, M\"oller, Brainbridge, and Habegger to prove finiteness results for Teichm\"uller curves \cite{Mc7, McM:spin,   M3, BM, BHM}. McMullen classified all orbit closures in genus 2 \cite{McMullen_classification}.\ann{\textbf{Referee:} ``(one problem remains open ...)" could you please be more
specific or delete the line?\\\textbf{Authors:} Line deleted.} %(one problem remains open concerning square-tiled surfaces)

\bold{Detailed statements.}
Let $\cM$ be an affine invariant submanifold, equipped with the flat bundles of absolute cohomology $H^1$ or relative cohomology $H^1_{rel}$.
The fibers of these bundles are the cohomology groups of the Riemann surfaces parametrized by $ \cM $.
The topological trivializations of the bundles lead to a flat connection, and parallel transport along the orbits of $ \GL_2^+(\bR) $ gives the Kontsevich--Zorich cocycle. 

Suppose now that $ V $ is a subbundle of $ H^1 $ or $ H^1_{rel} $, for instance $ V $ could be the entire bundle.
The algebraic hull of the Kontsevich--Zorich cocycle on $V$ is defined to be the smallest linear algebraic group such that there exists a measurable choice of basis in each fiber of $V$ such that all matrices obtained by parallel translation along $\GL_2 ^+(\bR)$ orbits in $\cM$  give matrices in that group. It is a nontrivial fact that the algebraic hull is well-defined up to conjugacy. We denote the algebraic hull by $A_V(\cM)$, or, when $\cM$ is clear from context, just $A_V$, and similarly we denote the Zariski closure of monodromy as $G_V(\cM)$ or just $G_V$. 

Let $p:H^1_{rel}\to H^1$ denote the forgetful map from relative to absolute cohomology.
For any flat subbundle $V\subset H^1_{rel}$, it is known that $p(V)$ is the direct sum of simple subbundles \cite{AEM}.

Consider the tautological bundle $ T\subset H^1_{rel}$, defined as the  span of the real and imaginary parts of the holomorphic $ 1 $-form.  The bundle $T$ is a fundamental example of a bundle which is $ \GL_2^+(\bR) $-invariant, but not flat (unless $ \cM $ is a \Teichmuller curve). We call $p(T)$  the tautological bundle of $H^1$.

\begin{thm}\label{T:AlgHull}
	Let $\cM$ be an affine invariant submanifold and let $V$ be any flat subbundle of $H^1$ or $H^1_{rel}$. 
	The algebraic hull $A_V$ is the stabilizer of the tautological plane in the Zariski closure of monodromy $G_V$.
	
	If $V$ does not contain the tautological plane, then the algebraic hull coincides with the Zariski closure of monodromy.
\end{thm}

Recall that any affine invariant submanifold $ \cM $ has a tangent bundle $ T\cM\subset H^1_{rel}$, and we may also consider its image $ p(T\cM)\subset H^1$  in absolute cohomology.
If the equations defining $ \cM $ have entries in some number field larger than $\bQ$, then the Galois conjugates of the above bundles give additioanal flat subbundles.

The next result follows from Theorem~\ref{T:AlgHull} and additional results on Zariski closure of monodromy.
Even for strata of abelian differentials, the result gives new, nontrivial information on the algebraic hull.

\begin{thm}\label{T:MoreAlgHull}
	Let $\cM$ be an affine invariant submanifold.\leavevmode
	\begin{enumerate}
		\item The algebraic hulls $A_{p(T\cM)}$ and $A_{T\cM}$ are the full group of endomorphisms that respect the symplectic form, the tautological plane, and, in the case of $T\cM$, the kernel of the map $p$ from relative to absolute cohomology. 
		\item The algebraic hull of a nontrivial Galois conjugate of $p(T\cM)$ or $T\cM$ is the full group of endomorphisms that respect the symplectic form, and, in the case of $T\cM$, the kernel of the map $p$ from relative to absolute cohomology.
	\end{enumerate}
\end{thm}

An endomorphism respects $\ker(p)$ if it acts as the identity on $\ker(p)$; in that case it respects the symplectic form if the induced map on the image in $ H^1 $ preserves the symplectic form.

The next results will be essential for proving finiteness statements on orbit closures.

\begin{thm}\label{T:equidistribution} 
	Fix an affine invariant submanifold $\cM$, and $V\subset H^1$ or $ V\subset H^1_{rel} $ a flat subbundle over $ \cM $.
	For any affine invariant submanifold $ \cM'\subset \cM $ we have the containment of algebraic hulls $A_{V}(\cM')\subset A_{V}(\cM)$.
	
	Furthermore, there is a finite union $\cB$ of proper affine invariant submanifolds of $\cM$ such that if $\cM'$ is not contained in $\cB$, then
	 $A_{V}(\cM')$ and  $A_{V}(\cM)$ are equal up to finite index and compact factors. 
\end{thm}

The second statement implies eventual agreement (up to finite index and compact factors) of algebraic hulls for infinite sequences of manifolds $ \cM_i $ equidistributing inside $ \cM $.

The locus $\cB$ is contained in a locus analogous to one where the second fundamental form of the Hodge bundle fails to have full rank. If the algebraic hull of $\cM$ is connected and has no compact factors, then the theorem gives that the algebraic hulls of $\cM$ and $\cM'$ are exactly equal. This is the case when $\cM$ is a stratum, and gives the following consequence, where we restrict to square-tiled surfaces because they are abundant and much studied.

\begin{cor}
	Let $\cM_i$ be a sequence of closed $\GL_2^+(\bR)$ orbits of genus $g$ square-tiled surfaces, and assume $\cM_i$ equidistributes to a stratum. Then $A_{H^1}(\cM_i)=Sp(2)\times Sp(2g-2)$ for all $i$ sufficiently large. 
\end{cor}

The key finiteness statement below involves the notions of rank and degree of an affine invariant manifold.
Recall that an affine manifold is cut out by linear equations in period coordinates of an ambient stratum, and has a tangent space.
There is a smallest field such that the coefficients of the linear equations can be chosen in it, and it is called the field of affine definition.
It is a number field and its degree over $ \bQ $ is called for brevity the degree of the manifold.
The rank of the manifold is defined as half the dimension of the projection of its tangent space to absolute cohomology.
For strata, the rank is the genus of the parametrized Riemann surfaces, and for \Teichmuller curves the rank is $ 1 $.  
More details on the field of affine definition can be found in \cite{AW_field}.

\begin{thm}\label{T:Fin}
In each stratum of Abelian differentials, all but finitely many affine invariant submanifolds have rank 1 and degree at most 2.  

In each genus there is a finite union of rank 2 degree 1 affine invariant submanifolds $\cM$ such that all but finitely many of the affine invariant submanifolds of rank 1 and degree 2 are a codimension 2 subvariety of one of these $\cM$. 
\end{thm}

A special case of of the first statement in Theorem \ref{T:Fin} is the following. 

\begin{cor}
In each genus, there are only finitely many Teichm\"uller curves with trace field of degree greater than 2. 
\end{cor} 

Affine invariant submanifolds of rank 1 and degree 1 consist of branched covers of tori, and these are dense in every stratum.

The strata  in genus 2 and  Prym loci in genus 3, 4, 5  are examples of rank $ 2 $  degree $ 1 $ affine invariant submanifolds; these are known to contain dense sets of codimension 2 affine invariant submanifolds of rank 1 and degree 2 by independent work of McMullen and Calta for the case of genus 2 and work of McMullen for the Prym loci \cite{Ca, Mc, Mc2}. A new example of a rank $ 2 $  degree $ 1 $ affine invariant submanifold was discovered and shown to contain infinitely many rank $ 2 $  degree $ 1 $ affine invariant submanifolds  in \cite{MMW}, and one additional example is forthcoming in \cite{EMMW}.

\begin{thm}\label{T:H2}
	Any rank 2 degree 1 affine invariant submanifold $\cM$  contains a dense set codimension 2 affine invariant submanifolds of rank 1 and degree 2.
\end{thm}

Theorem \ref{T:H2} is a consequence of the phenomenon discovered by McMullen in \cite{Mc}. 

For some results related to ours, see \cite{Ham}.  

\bold{Additional applications.} Forthcoming work will use results of this paper to study totally geodesic submanifolds of Teichm\"uller space \cite{Wtot} as well as marked points and the illumination and security problems \cite{AW}.

\bold{Organization.}
Section \ref{S:alghull} gives general background on algebraic hulls, and Section \ref{sec:computing_alg_hull} proves Theorem \ref{T:AlgHull}. Section \ref{S:monodromy} gives results on Zariski closure of monodromy, which together with Theorem \ref{T:AlgHull} imply Theorem \ref{T:MoreAlgHull}. Section \ref{sec:alg_hull_limits} proves Theorem \ref{T:equidistribution}, which is applied in Section \ref{S:fin} to prove Theorem \ref{T:Fin}.  Section \ref{S:fin} also establishes Theorem \ref{T:H2}. The two appendices extend results of \cite{sfilip_ssimple} and show the equivalence of two different definitions of algebraic hull. 

\bold{Acknowledgements.}
The proof of Theorem \ref{T:H2} arose from conversations with Ronen Mukamel.
We are  grateful to Brian Conrad for pointing out the reference \cite{SGA3}, to Curtis McMullen for comments on a previous draft, and to Amir Mohammadi for helpful conversations about algebraic groups. 

This research was partially conducted during the period AW and SF served as Clay Research Fellows. The research of AE is partially supported by NSF grants DMS 0905912 and DMS 1201422. 

%%%%%%%%%%%%%%%%%%%%%%%%%%%%%%%%%%%%%%%%%%%
%%%%%%%%%%%%%%%%%%%%%%%%%%%%%%%%%%%%%%%%%%%
\section{Algebraic hulls}\label{S:alghull}
%%%%%%%%%%%%%%%%%%%%%%%%%%%%%%%%%%%%%%%%%%%
%%%%%%%%%%%%%%%%%%%%%%%%%%%%%%%%%%%%%%%%%%%

This section introduces algebraic hulls, in the form that they will be used in this paper. Our definition is slightly different from Zimmer's \cite[Ch. 4]{Zimmer_book}; in the appendix (see \S\ref{app:zimmer_alg_hull}) we will show it is equivalent.

\bold{Chevalley's Theorem.}
The motivation for the definition of algebraic hull used in this paper is provided by the following result \cite[Prop. 3.1.4]{Zimmer_book}, sometimes called Chevalley's Theorem.
Let $G$ be an affine algebraic group, acting faithfully on a vector space $V$.
This gives an inclusion $G\subset \GL(V)$.
Below, a \emph{standard tensor operation} refers to taking direct sums, duals, or tensor products, in any order and any finite number of times.

\begin{thm}
\label{thm:chevalley}
For any algebraic subgroup $H\subset G$, there exists a representation $\cS(V)$ constructed from $V$ by standard tensor operations, and a line $l\subset \cS(V)$, such that $H$ is the stabilizer of $l$ in $G$.

Similarly, for any \emph{reductive} algebraic\ann{\textbf{Authors:} Added word ``algebraic" twice.} subgroup $H\subset G$, there exists a representation $\cS(V)$ constructed from $V$ by standard tensor operations, and a vector $v \in \cS(V)$, such that $H$ is the stabilizer of $v$ in $G$.
\end{thm}

For example, the diagonal subgroup of $\SL_2(\bR)$ stabilizes the quadratic form $dx\cdot dy$, and the upper-triangular group stabilizes the horizontal axis in $\bR^2$.

The connected component of the identity of an algebraic group is not always an algebraic group. For example, the connected component of the identity of $GL_2(\bR)$ is the positive determinant matrices, and this is not an algebraic subgroup of $GL_2(\bR)$. Non-algebraic subgroups cannot arise as a stabilizer as in Theorem \ref{thm:chevalley}.
\ann{\textbf{Authors:} Added cautionary remark.}  

\bold{Cocycles.}
Suppose now that a group $ A $ has an ergodic probability measure-preserving action on a measure space $ (X,\mu) $.
A cocycle over this action is a vector bundle $ V\to X $ with a lift of the action of $ A $ to $ V $ by linear transformations on the fibers.
Below, the fiber of $ V $ above a point $ x\in X $ will be denoted $ V_x $.

\begin{definition}[Algebraic Hull]
	\label{def:alg_hull_tensor}
	The algebraic hull for a cocycle  $ V\to X $ over an $A$-action is the collection of all groups $ G_x\subset \GL(V_x) $ such that $ G_x $ is the largest group preserving the fibers of all the $ A $-invariant line subbundles of tensor power constructions on $V$.
	Similarly, the \emph{reductive} algebraic hull is the collection of the largest groups preserving, in each fiber, the fibers of $ A $-invariant sections of tensor power constructions on $V$.
		
	The groups $ G_x $ are defined for $ \mu $-a.e. $ x\in X $.
	The definition naturally extends to give measurable, continuous, real-analytic, etc., algebraic hulls, where the corresponding adjective is imposed on the line subbundles or sections defining the algebraic hull.
\end{definition}

\begin{remark}
	\label{rmk:alg_hull}
	\leavevmode
 \begin{itemize}
  \item [(i)] Definition \ref{def:alg_hull_tensor} defines the algebraic hull to be a collection of groups $ G_x\subset \GL(V_x) $ above $\mu$-a.e point $ x $.
  In Appendix~\ref{app:alg_hull} we show that the conjugacy class of $ G_x $ is independent of the choice of $ \mu $-generic point $x$ and is equal to Zimmer's definition of algebraic hull.

  \item [(ii)] Algebraic varieties satisfy the finite chain condition (i.e. the Noetherian property) so in the definition above only finitely many lines, resp. tensors, are required.
  In fact, by Chevalley's Theorem \ref{thm:chevalley} a single one suffices.

  \item [(iii)] From the results in \cite{AEM} the algebraic hull of the Kontsevich--Zorich cocycle on $H^1$ is automatically reductive.
 \end{itemize}
\end{remark}

%%%%%%%%%%%%%%%%%%%%%%%%%%%%%%%%%%%%%%%%%%%
%%%%%%%%%%%%%%%%%%%%%%%%%%%%%%%%%%%%%%%%%%%
\section{Computing the Algebraic Hull}
\label{sec:computing_alg_hull}
%%%%%%%%%%%%%%%%%%%%%%%%%%%%%%%%%%%%%%%%%%%
%%%%%%%%%%%%%%%%%%%%%%%%%%%%%%%%%%%%%%%%%%%

\subsection{Setup}
In the setting of the Kontsevich--Zorich cocycle, we have the action of $ \GL_2^+(\bR) $ on a stratum of translation surfaces, and an affine invariant submanifold $ \cM $ equipped with a Lebesgue class probability measure, invariant and ergodic under $ \SL_2(\bR) $.
Each bundle in the short exact sequence
\begin{align}
\label{eqn:H1bdl_ses}
0 \to \ker(p) \to H^1_{rel} \xrightarrow{p} H^1 \to 0
\end{align}
is equipped with a flat connection inducing the Kontsevich--Zorich cocycle for the action of $ \SL_2(\bR) $.
The bundles in \eqref{eqn:H1bdl_ses} are taken to be real, i.e. the fibers are given by cohomology with real coefficients. (Parts of Section \ref{sec:alg_hull_limits} and the first appendix are exceptions; when we must refer frequently to the Hodge decomposition it is more convenient to use the complexified bundles.)
By passing to a finite cover of the ambient stratum (where zeros are labeled), assume that the bundle $\ker(p)$ is in fact trivial.
Associated to the short exact sequence \eqref{eqn:H1bdl_ses} is the sequence of groups denoted
\begin{align}
\label{eqn:H1gp_ses}
0 \to \Hom \left(H^1,\ker(p)\right) \to \Aut\left(H^1_{rel}\right) \to \Sp(H^1) \to 0.
\end{align}
Here $\Aut(H^1_{rel})$ is the group of automorphisms of $H^1_{rel}$ which preserve (i.e. act as the identity on) $\ker(p)$ and act symplectically on $H^1$. 
It can be written as a semi-direct product of the two other groups, but only after a choice of splitting $H^1\to H^1_{rel}$.

An element $ \xi \in \Hom(H^1,\ker (p)) $ induces a unipotent linear automorphism of $ H^1_{rel} $  defined by $ v\mapsto v + \xi(p(v)) $.
The map $\xi \mapsto (v\mapsto v + \xi(p(v)))$ identifies $\Hom(H^1,\ker(p))$ with the subgroup of automorphisms of $H^1_{rel}$  acting trivially on $H^1$. 

\subsection{Equivariant sections and bundles}
\label{ssec:eqvrt_bdls}

Suppose $W$ is some bundle over $\cM$, such that the $\GL^+_2(\bR)$-action on $\cM$ lifts to $W$ (e.g. by parallel transport, if $W$ is a local system).
Throughout, for a point $m\in \cM$ the subscript $\bullet_m$ (e.g. $W_m$) will denote the fiber over this point.

A section $\phi$ of $W$ defined over $\cM$ is \emph{equivariant} if for all $g\in \GL^+_2(\bR)$ and $m\in \cM$ we have $\phi(gm)=g\cdot \phi(m)$.

\bold{Rigidity of equivariant sections.}
Recall that local period coordinates on a stratum $\cH$ are given by the cohomology group $H^1_{rel}$.
Similarly, local period coordinates on $\cM$ are given by the sublocal system $T\cM\subset H^1_{rel}$.
These are local systems with complex coefficients that have a real structure, with local coordinates denoted by $(\boldx,\boldy)$ where $\boldx$ and $\boldy$ are the real and imaginary parts.

Fix a basepoint $m_0\in \cM$.
Monodromy refers to parallel transport along loops based at $m_0$. The local systems $H^1, H^1_{rel}$, etc., may be identified in a neighbourhood of $m_0$ with their fibers at $m_0$.

Let $A(\boldx,\boldy)$ denote the area of the flat surface with coordinates $(\boldx,\boldy)$, which may be calculated as the symplectic pairing of $p(\boldx),p(\boldy)\in H^1$.
This function is $\SL_2(\bR)$-invariant and scales quadratically under simultaneous scaling of both coordinates by $\bR^\times$.

The next result restricts the local nature of equivariant sections of bundles obtained from $H^1$ by any tensor operations.

\begin{thm}[{\cite[Thm. 7.7]{sfilip_ssimple}}]
	\label{thm:phi_polynomial}
	Let $ \mathbf{H} $ be some tensor construction on $ H^1 $ or $ H^1_{rel} $ and suppose $\phi$ is a $\GL_2^+(\bR)$-equivariant measurable section of $\mathbf{H}$ over $\cM$.
	Then in neighbourhood of $m_0$ in $\cM$, there exist flat local sections $\{s_l\}$ such that a.e. on the neighbourhood 
	\begin{align}
	\label{eqn:loc_form_eqvt_section}
	\phi(\boldx,\boldy) = \sum_l s_l \cdot \frac{P_l(\boldx,\boldy)} {A(\boldx,\boldy)^{k_l}},
	\end{align}
	where $P_l$ are homogeneous polynomials of bidegree $(k_l,k_l)$ in the variables $(\boldx,\boldy)$.
	Moreover, the polynomials $P_l(\boldx,\boldy)$ are invariant under the $\SL_2(\bR)$-action.
\end{thm}

This is proven in \cite[Th. 7.7]{sfilip_ssimple} for tensor constructions on $H^1$, and in Appendix~\ref{app:an_poly} we show that tensor constructions on $H^1_{rel}$ may be handled following the same outline.

\bold{Extending the section.}
Consider the Grassmanian $ \Gr^\circ(2,T\cM_{m_0}) $ of real $ 2 $-planes in $ T\cM_{m_0} $ whose projection to $ H^1 $ is a symplectically non-degenerate $ 2 $-plane. Note that this Grassmanian is an open subset of the full Grassmanian of $ 2 $-planes.

Consider the map $(X,\omega) \mapsto \operatorname{span}(\Re(\omega), \Im(\omega))$ from a simply connected neighbourhood $U$ of  $m_0$ in $\cM$ to $ \Gr^\circ(2,T\cM_{m_0}) $. The fibers of this map   are  connected components of the intersection of $\GL_2(\bR)$ orbits with $U$. Thus we may say that the set of  $\GL_2(\bR)$ orbits near $m_0$ is locally modeled on $\Gr^\circ(2,T\cM_{m_0}) $.

Since the polynomials defining $\phi$ are $\SL_2(\bR)$ invariant, $\phi$ defines a function on the image of $U$ in  $ \Gr^\circ(2,T\cM_{m_0}) $.
We now see that this function can be extended to all of $ \Gr^\circ(2,T\cM_{m_0}) $.

\begin{prop}
	\label{prop:phi_ext}
	Let $ \mathbf{H}_{m_0} $ denote the trivial bundle over $ \Gr^\circ(2,T\cM_{m_0}) $.
	Then the expression in \eqref{eqn:loc_form_eqvt_section} defines an algebraic section $ \phi_{ext,m_0} $ of  $\mathbf{H}_{m_0}$ defined on the entire Grassmanian $ \Gr^\circ(2,T\cM_{m_0}) $. 
\end{prop}
\begin{proof}
The expression in \autoref{eqn:loc_form_eqvt_section} defines a function on the space of points  $(\boldx,\boldy)$ in $T\cM_{m_0}\otimes \bC$ that satisfy $A(\boldx,\boldy)\neq 0$. This function is $\GL_2^+(\bR)$ invariant. The Grassmanian $ \Gr^\circ(2,T\cM_{m_0}) $ is the quotient of the   space by $\GL_2^+(\bR)$. 
\end{proof}

\bold{The Grassmanian.}
Define  $G_{T\cM} \subset \GL(T\cM)$ to be the subgroup which acts as the identity on $ (\ker p) \cap T\cM$ and by symplectic transformations on $ p(T\cM) $. (Later we will show that $G_{T\cM}$ is the algebraic hull of $T\cM$, justifying the notation.) 

The group  $G_{T\cM}$ acts transitively on $\Gr^\circ(2,T\cM_{m_0})$, so 
\begin{align}
\label{eqn:Gr_GS}
\Gr^\circ(2,T\cM_{m_0}) = G_{T\cM}/\Stab_T
\end{align}
is a homogenous space, where $T$ denotes a $2$-plane $T\subset T\cM_{m_0}$. To describe the stabilizer $\Stab_{T}$, consider the symplectic-orthogonal decomposition $(p(T\cM))_{m_0} = p(T) \oplus p(T)^\perp$.
Then $\Stab_T$ surjects onto $\Sp(p(T))\times \Sp(p(T)^\perp)$ with kernel a unipotent subgroup.

\bold{Compatibility with monodromy.}
Because the section $ \phi_{ext,m_0} $ defined by equation \eqref{eqn:loc_form_eqvt_section} \ann{\textbf{Referee:} ``$ \phi_{ext,m_0} $ " is this extension of $\phi$ from equation (3.4)?\\\textbf{Authors:} Yes. Clarification added.} is polynomial, its equivariance properties extend to Zariski closures, as the next result shows.

\ann{\textbf{Referee:} is the cocycle that gives the action on $s_l$ a rational function?\\\textbf{Authors:} The action is by linear transformations on $\mathbf{H}_{m_0}\oplus T\cM_{m_0}.$ The linear action on $T\cM_{m_0}$ induces an algebraic action on $\Gr^\circ(2,T\cM_{m_0})$, as can be seen from the Plucker embedding. Hence we obtain an algebraic action on the total space $\mathbf{H}_{m_0}\times \Gr^\circ(2,T\cM_{m_0})$, which induces an algebraic action on sections by precomposition.}
\begin{prop}
	\label{prop:eqvt_Grass}
	Let $\phi$ be a $\GL_2^+(\bR)$-equivariant section of $\mathbf{H}$ over $\cM$.
	By Proposition~\ref{prop:phi_ext} it gives rise to an algebraic section $\phi_{ext,m_0}$ of the trivial bundle $\mathbf{H}_{m_0}$ over $\Gr^\circ(2,T\cM_{m_0})$.
	
	Then $ \phi_{ext,m_0} $ is equivariant for the simultaneous action of the Zariski closure of the monodromy inside $\GL\left(\mathbf{H}_{m_0}\right)\times \GL (T\cM_{m_0})$ (as given in \eqref{eqn:equiv_prop} below).
\end{prop}
\begin{proof}
	Given the local description of an equivariant section $\phi$ in \eqref{eqn:loc_form_eqvt_section}, consider its behavior under a change of chart.
	Both the sections $s_l$ and the coordinates $\boldx,\boldy$ will change according to the change of coordinates map.
	Going around a loop $\gamma$ in the affine manifold $\cM$ and comparing results, it follows that
	\begin{align}
	\label{eqn:equiv_prop}
	\sum_l s_l \cdot \frac{P_l(\boldx,\boldy)} {A(\boldx,\boldy)^{k_l}} 
	= \sum_l (\rho_{\mathbf{H}}(\gamma)s_l )\cdot \frac{P_l(\rho_{T\cM}(\gamma)\boldx, \rho_{T\cM}(\gamma)\boldy)} {A(\rho_{T\cM}(\gamma)\boldx,\rho_{T\cM}(\gamma)\boldy)^{k_l}}.
	\end{align}
	Here $\rho_\bullet(\gamma)$ denotes the monodromy matrix corresponding to $\gamma$.
	On $s_l$ it acts by the appropriate cocycle on $\mathbf{H}_{m_0}$ (denoted $\rho_{\mathbf{H}}$ for brevity); on $(\boldx,\boldy)$ it acts via the representation on $T\cM$.

	The equality of rational functions in \eqref{eqn:equiv_prop} holds for all $ \gamma $ in the monodromy, therefore it holds for all elements of the Zariski closure of monodromy.
\end{proof}

We can now prove the main result, Theorem~\ref{T:AlgHull}.

\begin{proof}[Proof of Theorem~\ref{T:AlgHull}.]
	From the definition of algebraic hull, it is clear that it is contained in the Zariski closure of monodromy.
	Indeed, any flat tensor is automatically $ \GL_2^+(\bR) $-invariant.
	Additionally, the tautological plane (in either $ H^1 $ or $ H^1_{rel} $) is also invariant under the $ \GL_2^+(\bR) $-action, so the algebraic hull must be contained in its stabilizer.
	In particular, the algebraic hull is contained in the intersection of these two groups.
	
	We need to show that conversely, the algebraic hull contains the intersection of the Zariski closure of monodromy and the stabilizer of the tautological plane.
	Suppose therefore that $ \phi $ is a global tensor in $ \mathbf{H} $ (some tensor construction on $ H^1 $ or $ H^1_{rel} $) and that $ \phi $ is $ \GL_2^+(\bR) $-equivariant.
	
	Fix a basepoint $ m_0\in \cM$. By Proposition~\ref{prop:phi_ext}, $\phi $ extends to an algebraic section $ \phi_{m_0,ext}$ of the trivial bundle $ \mathbf{H}_{m_0} \to \Gr^\circ(2,T\cM_{m_0}) $.
	Note that $ m_0 $ gives a point in $ \Gr^\circ(2,T\cM_{m_0}) $ corresponding to its $ \GL_2^+(\bR) $-orbit (i.e. the value of the tautological plane at $ m_0 $).
	By construction $ \phi_{m_0,ext}(m_0) = \phi(m_0) $.
	
	By Proposition~\ref{prop:eqvt_Grass} the section $ \phi_{m_0,ext} $ is equivariant for the Zariski closure of monodromy acting simultaneously on the Grassmanian and on the bundle over it, i.e. for any such $ \gamma $ we have
	\begin{align*}
	\phi_{m_0,ext}(\gamma \cdot p) = \gamma\cdot \phi_{m_0,ext}(p).
	\end{align*}
	Now if $ \gamma $ is also in the stabilizer of the tautological plane, the above equation becomes (for $ p=m_0 $) simply $ \phi_{m_0,ext}(m_0)= \gamma\cdot  \phi_{m_0,ext}(m_0) $.
	
	Using that $ \phi_{m_0,ext}(m_0) = \phi(m_0) $, this implies $ \phi(m_0) =\gamma \cdot \phi(m_0)  $.
	This invariance of $ \phi $ implies that the algebraic hull contains the intersection of the Zariski closure of monodromy and the stabilizer of the tautological plane.
\end{proof}

%%%%%%%%%%%%%%%%%%%%%%%%%%%%%%%%%%%%%%%%%%%
%%%%%%%%%%%%%%%%%%%%%%%%%%%%%%%%%%%%%%%%%%%
\section{Monodromy}\label{S:monodromy}
%%%%%%%%%%%%%%%%%%%%%%%%%%%%%%%%%%%%%%%%%%%
%%%%%%%%%%%%%%%%%%%%%%%%%%%%%%%%%%%%%%%%%%%

\subsection{Setup}
We keep the notation for bundles as in the beginning of Section~\ref{sec:computing_alg_hull}, in particular we shall use the short exact sequence of cohomology bundles \eqref{eqn:H1bdl_ses}. 

Consider a fixed affine invariant manifold $\cM$ in a stratum of flat surfaces $\cH$. Define $\ker(p)_\cM$ to be $\ker(p)\cap T\cM$. We have the short exact sequences 
\begin{align}
0 \to \ker(p)_\cM \to T\cM \xrightarrow{p} p(T\cM)\to 0
\label{eqn:TM_ses}
\intertext{of bundles and}
0 \to U \to G_{T\cM} \xrightarrow{p} G_{p(T\cM)} \to 0
\label{eqn:TM_gp_ses}
\end{align}
of monodromies.
The groups $G_{T\cM}$ and $G_{p(T\cM)}$ are the Zariski closures of monodromies on the corresponding bundles, while $U$ denotes the kernel of the projection.
Note that $U$ is naturally a subgroup of transformations that act by the identity when projected to $p(T\cM)$.
This last group is naturally identified with $\Hom(p(T\cM),\ker(p)_\cM)$, which will be abbreviated $U_\cM$.

In Section~\ref{ssec:monodromy_computation} we shall compute the Zariski closure of the two monodromy groups.
Namely, we will see that on absolute cohomology we have $G_{p(T\cM)}=\Sp(p(T\cM))$, i.e. we get the full symplectic group.
On relative cohomology, we will see that the group is as large as it can be, namely the kernel $U$ is all of $U_\cM$.

Additionally, we will see that on bundles other than those coming from $\cM$, the monodromy is ``decoupled'' from that on the tangent bundle to $\cM$.
A precise statement is Proposition~\ref{prop:monodromy_decoupling}.

\begin{remark}[On Zariski closures]\ann{\textbf{Referee:} please make use of standard terminology -- given a
subset, say of $\bA^n$, one usually defines the Zariski closure of this subset and then discusses various rationality questions e.g. the field of definition of the variety thus obtained.\\\textbf{Authors:} Added clarification.}
	Throughout this paper, Zariski closures are taken with respect to $ \bR $, as opposed to $ \bQ $ (which could be larger).
	Concretely, the Zariski closure of a subset of $GL_n(\bR)\subset \bR^{n^2}$ is the intersection of $GL_n(\bR)$ and the zero locus in $\bR^{n^2}$ of all \emph{real} polynomials vanishing on the set.  
\end{remark}

\subsection{Monodromy of the Kontsevich--Zorich cocycle}
\label{ssec:monodromy_computation}

\bold{Monodromy and absolute cohomology.}
It follows from \cite{AEM} that the Kontsevich--Zorich cocycle on $H^1$ or any of its tensor powers is reductive, i.e. any invariant bundle has a complement.
This gives the decomposition of $H^1$ over $\cM$
\begin{align}
\label{eqn:ssimple_dec}
H^1 = p(T\cM) \oplus \left(\bigoplus_\iota p(T\cM)_{\iota}\right) \oplus V.
\end{align}\ann{\textbf{Referee:} please recall which number field is relevant here.\\\textbf{Authors:} Done.}
The bundles $p(T\cM)_{\iota}$ are Galois-conjugates of $p(T\cM)$ and $V$ is the symplectic (as well as Hodge) orthogonal to the other spaces. There is one $p(T\cM)_{\iota}$ for each non-identity embedding of the field of affine definition of $\cM$ to $\bR$; see \cite{AW_field} for more details.
The list of possible monodromy groups, up to compact factors is given in \cite[Thm 1.2]{sfilip_zero}.
In particular, the next result holds.

\begin{thm}[{\cite[Corollary 1.7]{sfilip_zero}}]
	\label{thm:Zar_cl_pTM}
	The Zariski closure of monodromy on $p(T\cM)$ or any of its Galois conjugates $ p(T\cM)_\iota $ is the full symplectic group $\Sp(p(T\cM))_\iota$. 
\end{thm}\ann{\textbf{Referee:} Is the Zariski closure of the monodromy on compliment of $V$, from equation (4.2), a restriction of scalars group? (I could not find this explicated in the reference provided either. Moreover, statements like the algebraic hull has only one non-compact factor on an $\mathbb{R}$-irreducible piece appear in the reference -- this was confusing to me: shouldn't the
group be $\bR$-almost simple on an $\bR$-irreducible piece?)\\ \textbf{Authors:} Monodromy on $\bigoplus_\iota p(T\cM)_\iota$ can be viewed as a restriction of scalars group from monodromy on $p(T\cM)$. \\
Regarding your parenthetical question: For examples with multiple factors, see Remark \ref{R:compactfactors}. The product of two groups can have irreducible representations.}

\begin{remark}
This particular consequence of \cite{sfilip_zero} can also be derived as follows, assuming familiarity with \cite{Wcyl}. By \cite[Lemma 4.6]{MirWri}, $\cM$ contains a surface with an  equivalence class of $\cM$-parallel cylinders all of whose moduli are rationally related. Let $\alpha_i$ be the core curves of the cylinders in the given equivalence class. The twist on this equivalence class gives a closed loop in $\cM$, whose monodromy on $H_1$ is a composition of powers of the Dehn twists $\gamma\mapsto \gamma+ \alpha_i\langle \alpha_i , \gamma\rangle$. On $p(T\cM)^\vee$, which may be viewed as a subbundle of $H_1$, all the $\alpha_i$ are collinear by the definition of $\cM$-parallel, hence the monodromy has the form  $\gamma\mapsto \gamma+ C \alpha_1\langle \alpha_1 , \gamma\rangle$. A short argument of Kazhdan-Margulis gives that any subgroup of the symplectic group that is totally irreducible and contains such a transformation must be Zariski dense, compare to \cite[page 250]{Etale}. By \cite{AW_field}, the monodromy of $p(T\cM)$ is totally irreducible. 
\end{remark}

% AW: PLEASE DO NOT DELETE THE FOLLOWING COMMENT. 

%\begin{lem}[Kazhdan-Margulis]
%Let $G$ be a closed subgroup of $Sp(V)$ that contains a transformation of form $\gamma\mapsto \gamma+ \alpha\langle \alpha, \gamma\rangle$ and acts totally irreducibly on $V$. Then $G=Sp(V)$. 
%\end{lem}
%
%\begin{proof}
%Let $W$ be the set of $\alpha\in V$ for which the Lie algebra $Lie(G)$ contains the transformation $N(\alpha)$ defined by $\gamma\mapsto \alpha\langle \alpha , \gamma\rangle$. The set $W$ is closed under scaling and by assumption is non-zero. The algebraic identity 
%$$ \langle \alpha, \beta \rangle  N(\alpha+\beta) =  \langle \alpha, \beta \rangle  N(\alpha)+ \langle \alpha, \beta \rangle N(\beta) + [N(\alpha), N(\beta)]$$
%implies that if $\alpha, \beta\in W$ then either $\langle \alpha, \beta\rangle = 0$ or $\alpha+\beta\in W$. 
%
%$W$ is invariant under $G$. Let $U$ be a maximal vector subspace of $W$. If $U$ is symplectic, the proof is complete, because the orbit of $U$ under $G$ must be finite because any two elements of the orbit are symplectically orthogonal. 
%
%If $U$ is not symplectic, because the orbit of $U$ must span $V$, there must be some $w\in U$ and $g\in G$ such that $ gw\notin U^\perp$ and $gw\notin U$. (Indeed, take $u\in U \cap U^\perp$. There must be some $w\in U$ and $g\in G$ so that $\langle gw, u\rangle \neq 0$.) Since $gw\in W$ this gives a contradiction. 
%\end{proof}

\bold{Monodromy and relative cohomology.}
The monodromy on $T\cM$ surjects onto monodromy on $p(T\cM)$ and we would like to understand the unipotent kernel $U$ (see \eqref{eqn:TM_gp_ses}).
Recall also that the kernel sits inside the linear transformations that act by identity on $p(T\cM)$, which is naturally identified with $U_\cM:=\Hom(p(T\cM),\ker(p)_\cM)$.
For the Galois-conjugate bundles, we can similarly define $U_{\cM,\iota}:=\Hom(p(T\cM)_\iota,\ker(p)_\cM)$.

\begin{prop}
	\label{prop:Zar_cl_TM}
	The Zariski closure of monodromy on $T\cM$ is $U_\cM \rtimes \Sp(p(T\cM))$.
	In other words, the unipotent part is as large as it can be.
	
	The same statement holds for the Galois-conjugate bundles: the Zariski closure of monodromy is $ U_{\cM,\iota}\rtimes \Sp(p(T\cM)_\iota) $.
\end{prop}
\begin{proof}
	First, observe that $U_\cM$ has a natural action of the monodromy $\Sp(p(T\cM))$.
	The kernel $U\subset U_\cM$ is invariant under this action, therefore $U = \Hom(p(T\cM),S)$ for some subspace $S\subset \ker(p)_\cM$. 
	We will now see that $S$ in fact equals $\ker(p)_\cM$, thus establishing the claim.
	
	Suppose therefore that $S$ is such that the unipotent part of the monodromy is contained in $\Hom(p(T\cM),S)$.
	We will construct a flat subbundle $E_S\subset T\cM$ as follows.
	
	Fix a point $m_0\in \cM$ and consider a lift $L$ of $p(T\cM)_{m_0}$ to $T\cM_{m_0}$, i.e. a subspace $L\subset T\cM_{m_0}$ such that $p$ is an isomorphism from $L$ to   $p(T\cM)$.
	Define the fiber of $E_S$ at $m_0$ to be $(E_S)_{m_0}=\spn(L,S)$.
	
	Extend now $E_S$ by parallel transport of the fiber at $m_0$ to all of $\cM$.
	By the assumption on monodromy that its unipotent part lies in $\Hom(p(T\cM),S)$, we see that this gives a well-defined extension.
	Note also that $E_S\cap \ker(p)=S$ by construction and $E_S$ is not contained in $\ker(p)$.
	
	However, by \cite[Theorem 7.4]{AW_field}\ann{\textbf{Referee:} please make the reference to the paper [Wri14] more explicit.\\\textbf{Authors:} We added a theorem number.} there are no proper flat subbundles of $T\cM$ other than those contained in $\ker(p)$, therefore $E_S=T\cM$ and thus $S=\ker(p)\cap T\cM$.
	
	The statement for the Galois-conjugate bundles now follows by noting that the dimension of the unipotent part does not change under Galois conjugation, and it is maximal.
\end{proof}

\bold{Decoupling monodromies.}
Let  $ \Gamma $ denote the (orbifold) fundamental group of $ \cM $.
Monodromy representations are denoted $\rho_V:\Gamma\to \GL(V)$, where $V$ is an appropriate bundle (e.g. $T\cM, p(T\cM), H^1$, etc) and their Zariski closures are denoted $G_V$.

\begin{prop}
	\label{prop:monodromy_decoupling}
	Suppose that $W$ is a flat irreducible subbundle of $H^1$ other than $p(T\cM)$, and let $G_W$ be the Zariski closure of the monodromy $\rho_W(\Gamma)$.
	\begin{itemize}
		\item[(i)] The Zariski closure of monodromy on $p(T\cM) \oplus W$ is $\Sp(p(T\cM))\times G_W$.
		\item[(ii)] The Zariski closure of monodromy on $T\cM\oplus W$ is $G_{T\cM}\times G_W$.
		\item[(iii)] Let $ \widetilde{W}:=p^{-1}(W) $ denote the preimage of $ W $ in relative cohomology, and assume $ G_{\tilde{W}} $ is the Zariski closure of monodromy on it.
		Then the Zariski closure of monodromy on $ T\cM \oplus \widetilde{W} $ is $ G_{T\cM}\times G_{\widetilde{W}} $.
	\end{itemize}
\end{prop}

\begin{proof}
	To prove (i),\ann{\textbf{Referee:} Please see my question above about Theorem 4.5\\\textbf{Authors:}  $W$ might not be a Galois conjugate of $p(T\cM)$, and therefor we are not considering restriction of scalars groups.} let $H:=\overline{\rho_{p(T\cM)\oplus W}(\Gamma)}$ be the Zariski closure of the monodromy.
	By assumption $H\subset \Sp(p(T\cM))\times G_W$ surjects when projected to either component.
	
	Consider the kernel $K\subset H$ of the surjection $H \onto G_W$.
	Note that projecting $K$ to $\Sp(p(T\cM))$ embeds it as a normal subgroup.
	Indeed, $K$ is normal in $H$, and $H$ surjects onto $\Sp(p(T\cM))$.
	
	Since the symplectic group is simple, $ K $ is either the full group (in which case we are done) or trivial.
	Suppose that $ K $ is trivial.
	
	Then $ H $ is the graph of an isomorphism $ \Sp(p(T\cM)) \to G_W $.
	Therefore $ G_W $ is isomorphic to a symplectic group, and by the classification in \cite[Thm. 1.2]{sfilip_zero} a symplectic group can only occur in the standard representation.
	Thus the isomorphism of monodromy groups gives also an isomorphism of the flat bundles $p(T\cM)$ and $W$.
	It follows that $ p(T\cM) $ and $ W $ have the same Lyapunov exponents.
	This is a contradiction, since $p(T\cM)$ has top Lyapunov exponent $1$, and all other exponents in $H^1$ are strictly smaller by a result of Forni \cite{Forni_deviations}. (It is also possible to derive a similar contradiction by directly using the representation theory of the symplectic group instead of invoking \cite[Thm. 1.2]{sfilip_zero}.)
	
	To prove (ii), let again $H\subset G_{T\cM}\times G_W$ be the Zariski closure of the monodromy.
	Recall the unipotent radical of $G_{T\cM}\times G_W$ is $G^u_{T\cM}$.
	Letting $K$ be the kernel of the map $H\to \Sp(p(T\cM))\times G_W$, it follows that $K\subset G^u_{T\cM}$.
	Since $H$ surjects onto $G_{T\cM}$ by Proposition~\ref{prop:Zar_cl_TM}, it follows that $K$ surjects onto $G^u_{T\cM}$.
	Using part (i), it now follows that $H$ is all of $G_{T\cM}\times G_W$.
	
	To prove (iii), let again $ H\subset G_{T\cM}\times G_{\widetilde{W}} $ be the Zariski closure of monodromy; we want to show $ H $ is the entire product.
	By assumption $ H $ surjects onto both $ G_{\widetilde{W}} $ and $ G_{T\cM} $.
	Let $ K\subset H $ be the kernel of the surjection to $ G_{\widetilde{W}} $.
	As before, $ K $ viewed as a subgroup in $ G_{T\cM} $ is normal and surjects onto $ \Sp(p(T\cM)) $.
	Such subgroups are in bijection with flat bundles $ E\subset T\cM $ which surject onto $ p(T\cM) $, but from \cite[Thm 5.1]{AW_field} we must have $ E=T\cM $, i.e. $ K=G_{T\cM} $.
	To see the bijection between subgroups and subbundles, note that the group is given by automorphism of the fibers which preserve the subbundle (and act symplectically on the quotient).
\end{proof}

%%%%%%%%%%%%%%%%%%%%%%%%%%%%%%%%%%%%%%%%%%%
%%%%%%%%%%%%%%%%%%%%%%%%%%%%%%%%%%%%%%%%%%%
\section{Algebraic Hulls along limits}
\label{sec:alg_hull_limits}
%%%%%%%%%%%%%%%%%%%%%%%%%%%%%%%%%%%%%%%%%%%
%%%%%%%%%%%%%%%%%%%%%%%%%%%%%%%%%%%%%%%%%%%

Recall that by the results of \cite{EM, EMM}, given any infinite sequence of affine invariant manifolds $\{\cM_i\}$ in a fixed stratum, there exists an affine manifold $\cM$ and a subsequence $\{\cM_{n_i}\}$ such that $\cM_{n_i}\subseteq \cM$, and Lebesgue measure on $\cM_{n_i}$ tends to Lebesgue measure on $\cM$.
In particular, the $\cM_{n_i}$ become dense in $\cM$.

This section will establish that the algebraic hulls of the manifolds $\cM_{n_i}$ and the algebraic hull of $\cM$ eventually agree, up to finite index and compact factors.
We give two proofs -- one of a general ergodic-theoretic flavor, and a second one, based on \HT tensors.
Both proofs imply that outside a finite collection of affine submanifolds $ \cB\subset \cM $, any other affine submanifold $ \cM'\subset \cM $ has the same algebraic hull as $ \cM $, up to compact factors (assuming $ \cB\cap \cM' = \emptyset $).
The proof via \HT tensors gives further information on the locus $ \cB $, whereas the ergodic-theoretic one only implies its existence.

\bold{Finite index, compact factors, and Forni subspaces.}
For the purposes of this section, two groups $ G_1, G_2 $ (contained in the same ambient $ \GL_n $) \emph{agree up to finite index} if their connected components of identity are the same.
The two groups agree \emph{up to compact factors} if there is a common normal subgroup $ N $ such that each of the quotients $ G_i/N $ is compact.

\begin{remark}\label{R:compactfactors}\leavevmode
\begin{enumerate}
	\item Compact factors in the monodromy of the Kontsevich--Zorich cocycle arise in particular when a sub-VHS $V$ of $H^1$ may be written as the tensor product of a weight 1 VHS with Zariski closure of monodromy $G$ and a weight 0 VHS with Zariski closure of monodromy $K$, which must be compact.
	In this case the monodromy of $V$ is contained in $K\times G$, which may act via an irreducible representation (some examples can be found in \cite{FFM}). In the case where the weight 1 VHS is trivial one calls $V$ a Forni subspace \cite{FMZ_zero, Forni_deviations}.
	The results below allow for the possibility that $\cM$ has Zariski closure of monodromy $G\times K$ but $\cM_i$ has monodromy contained in $G\times K'$ with $K' \subsetneq K$. 
	\item \ann{\textbf{Authors:} We added part (2) to Remark \ref{R:compactfactors}. It will be relevant to  the updated proofs of Theorem \ref{thm:limit_absolute}.} For an algebraic group $G$, denote by $G^\zarcon{}$ the minimal normal algebraic subgroup of $G$ such that $G/G^\zarcon{}$ is compact.
	Note that in particular $G^\zarcon{}$ is connected in the Zariski topology.
	To prove that two algebraic groups $G_1,G_2$ agree up to finite index and compact factors, it suffices to check that $G_1^\zarcon{}$ and $G_2^\zarcon{}$ agree.

	The algebraic hull of any factor of the Kontsevich--Zorich cocycle is semisimple, up to compact factors.
	Indeed, from \cite{AEM} it follows that the hull is reductive and from the classification in \cite[Thm.~1.2]{sfilip_zero} it follows that all non-compact factors have to be semisimple.
	In particular, the algebraic hull cannot contain any non-compact abelian factors. 
	
	It then follows that for the algebraic hull $G$ of any factor of the Kontsevich--Zorich cocycle, the group $G^\zarcon{}$ is semisimple.
\end{enumerate}

\end{remark}

\subsection{The case of absolute cohomology}

\begin{thm}
  \label{thm:limit_absolute}
  Suppose $\{\cM_i\}$ is an infinite sequence of affine submanifolds of $\cM$ that equidistribute towards another affine manifold $\cM$.
  
  Let $G_i$ and $G$ be the algebraic hulls of $\cM_i$ and $\cM$ respectively, for the $\GL_2^+(\bR)$-action on the absolute cohomology bundle $H^1$.
  \begin{itemize}
   \item [(i)] We have $G_i\subseteq G$ for all
   sufficiently large $i$.
   \item [(ii)] There exists $N\geq 0$ such that for all $i\geq N$, the groups $G_i$ and $G$ agree up to finite index and compact factors.
  \end{itemize}
\end{thm}

\begin{remark}
 By Theorem \ref{thm:Zar_cl_pTM}, the Zariski closure of monodromy for $p(T\cM)$ and its Galois conjugates do not have compact factors, and so the algebraic hull of any of these bundles for $\cM_i$ and $\cM$ are exactly equal for $i$ sufficiently large. (Later this will also apply to $T\cM$.)
  \end{remark}\ann{\textbf{Referee:} Please see
my question above about Theorem 4.5\\\textbf{Authors:} We added a reference to Theorem \ref{thm:Zar_cl_pTM}.}
 
 \subsection{First proof of Theorem \ref{thm:limit_absolute}}
We begin by recalling some useful preliminaries from ergodic theory.

Suppose we have a bundle $ P_F \to X $ over some space $ X $, with fiber $ F $ and structure group $ G $.
In other words, locally on $ X $ we have an isomorphism $ P_F\vert_{U_i} \cong U_i\times F $ (where $ U_i\subset X $) and the gluing maps on overlaps are given by maps $ U_i\cap U_j \to G $.

Suppose now that $ X $ carries an action of a group $ A $, and the action lifts to $ P_F $ by $ G $-maps, i.e. after local trivialization of the bundle, the maps between fibers are in $ G $. (This is independent of the trivialization, since gluings are in $ G $ as well.)
Suppose next that $ s:X\to P_F $ is an $ A $-equivariant section, i.e. $ s(a\cdot x) = a\cdot s(x) $.
Then $ s $ descends to a \emph{map}
\begin{align}
	\label{eqn:section_descent}
s : X\to F/G
\end{align}
where $ F/G $ is the space of $ G $-orbits on $ F $.

The next results, due to Borel--Serre in the algebraic case, and Margulis and Zimmer for measures, give control over spaces of $ G $-orbits (see \cite[Sec. 3.2]{Zimmer_book}).
Throughout, we consider the real points of the corresponding algebraic varieties.

\begin{prop}
	\label{prop:separability}
	Let $G$ be a real-algebraic group, and $V$ an algebraic variety with a $G$-action.
	
	Then the space of $G$-orbits on $V$, with its induced topology, is countably separated (\cite[Def. 2.1.8]{Zimmer_book}).
	Moreover, for any two $G$-orbits in $V$, there is a closed $G$-invariant set which contains one, but not the other.
	
	The same separability properties hold for the space of probability measures on $V$ with the weak topology, for the induced action of $G$.
\end{prop}

\begin{proof}[First proof of Theorem \ref{thm:limit_absolute}] \ann{\textbf{Referee:} Perhaps it is useful to remind the reader $G_i$'s have bounded degree, e.g. $G_i$'s are semisimple -- this seems to be used to bound the dimension of representations in Chevalley's theorem.\\\textbf{Authors:} We don't have to remark on that, because there are only finitely many possibilities for the isomorphism type of $G_i^\zarcon{}$, and each $\Hom(G_i^\zarcon{}, G^\zarcon{})$ is a variety.}
  Part (i) is immediate from \cite[Thm. 1.5]{sfilip_ssimple} -- the measurable and continuous (in fact, real-analytic) algebraic hulls of $\cM$ coincide.
  Indeed, the cited result implies that any measurable invariant tensor is necessarily continuous, and is thus well defined on $ \cM_i\subset \cM $.
  
  Part (ii) is proved by contradiction.
  First, by passing to a finite cover of $\cM$, we can assume that $G$ is Zariski-connected, in particular an irreducible representation of $G$ is strongly irreducible, i.e. it does not contain a proper finite collection of subspaces permuted by $G$.
  Suppose now that there exists a subsequence of the $\cM_i$ such that we have the strict inclusion $G_i^{\zarcon}\subsetneq G^{\zarcon}$, where for a group $H$ the subgroup $ H^{\zarcon} $ was defined in Rmk.~\ref{R:compactfactors}(2) as the smallest normal subgroup such that $H/H^\zarcon{}$ is compact.

  Chevalley's Theorem associates to the groups $G_i^\zarcon\subset G$ lines $ l_i $ inside $G$-representations $V_i$ such that $G_i^\zarcon$ is the stabilizer of $ l_i $.\ann{\textbf{Authors:} The proof has been restructured and shortened. First, it deals with issues of irreducibility and connected components earlier in the proof, and where possible at the level of representations instead of bundles.
  Second, it introduces the groups $G^\zarcon{}$ (see \ref{R:compactfactors}) which are in fact semisimple, not just reductive, so that the quoted theorems apply.}
  There are only finitely many isomorphism types of groups $ G_i^\zarcon $, since they are semisimple and connected (see Rmk.~\ref{R:compactfactors}(2)).
  By \cite[Thm. XXIV 7.3.1(i)]{SGA3} the homomorphisms $ \Hom(G_i^\zarcon, G^\zarcon) $ form a variety of finite type, so in particular $ \Hom(G_i^\zarcon, G^\zarcon) $ has finitely many components.
  % DO NOT DELETE
  % Comment added by AW on Nov 8, 2017. If G_i is reductive but not semisimple, then Him(G_i, G) can have infinitely many connected 
  % comments. Ex, Hom(C^*, C^*) has infinitely many components, given by z->z^n. 
  The proof of Chevalley's Theorem shows that all subgroups parameterized by a given component can be obtained as stabilizers of lines in a fixed representation $V$.
  So passing to a further subsequence (still denoted $G_i^\zarcon$), we may assume that all the $ l_i $ occur inside one single representation $ V $.

  To get a contradiction to $G_i^\zarcon{}\subsetneq G^\zarcon{}$ for all sufficiently large $i$, it suffices to prove the following.
  Let $V=V'\oplus V''$ where $V'$ is $G$-irreducible and $V''$ is $G$-invariant.
  If $l_i$ is contained in $V''$ for all sufficiently large $i$, then we apply the reasoning to $V''$.
  Therefore, suppose that $l_i$ projects non-trivially to $V'$ along a subsequence.
  We will show $G$ acts on $V'$ via a compact group. 
  If $G$ acts on each irreducible $V'\subset V$ via a compact group, then $G$ acts on $V$ via a compact group, and so we will be able to conclude that $G_i^\zarcon{} = G^\zarcon{}$. (We must also have that $G_i$ acts on $V$ via a compact group, since $G_i\subset G$.)
    %Therefore, we can repeat the argument by considering the kernel of $G\to \GL(V')$ and eventually conclude that $G_i^\zarcon{} = G^\zarcon{}$

  Assume that $l_i$ projects non-trivially to $V'$ for a subsequence, which we take again to be $l_i$.
  From now on, identify all the groups with their images in $\GL(V')$ and assume that $l_i\subset V'$.
  Then the orbit $G_i\cdot l_i \subset \bP(V')$ is identified with $G_i/G_i^\zarcon{}$ and carries a natural $G_i$-invariant probability measure $\eta_{i,mod}$, since the quotient $G_i/G_i^\zarcon{}$ is compact.

  Associate to the $G$-representation $V'$ the vector bundle $E'\to \cM$; since $V$ arises as a subrepresentation in tensor construction on the natural representation of $G$ on $H^1$, $E'$ is itself a subbundle of such a natural tensor construction.
  Over each $\cM_i$ define the measure $\eta_i$ on $\bP(E')$ which is the product of Lebesgue measure on $\cM_i$ with the model measure $\eta_{i,mod}$ in the fiber direction.
  By construction, the probability measure $\eta_i$ is invariant under $\SL_2(\bR)$.

  Let now $\eta$ be any weak limit of the $\eta_i$; it will still be invariant under $\SL_2(\bR)$ and now project to Lebesgue measure on $\cM$.
  Denote by $\mathscr{P}(\bP(E'))$ the bundle of probability measures on the fibers of $\bP(E')$.
  Then the measure $\eta$, via its disintegration, gives a section $s:\cM \to \mathscr{P}(\bP(E'))$.

By \eqref{eqn:section_descent}, the section $s$ descends to a map
$[s]:\cM\to G\backslash\mathscr{P}(\bP(V'))$ where
$\mathscr{P}(\bP(V'))$ is the space of probability measures on the
projectivization of the $G$-representation $V'$. Because the space
of probability measures divided by the $G$-action is countably
separated (Proposition~\ref{prop:separability}), by
\cite[Proposition 2.1.10]{Zimmer_book} it follows that $s$ takes
values in a single $G$-orbit, Lebesgue-a.e. on $\cM$.
Therefore, for
Lebeague a.e.\ $x \in \cM$, the measure $\eta_x$ given by the
disintegration of $\eta$ along the fiber $E'_x \approx V'$ is given by
\begin{displaymath}
\eta_x = \psi(x)_* \eta_0,
\end{displaymath}
where $\psi(x) \in G$, and $\eta_0$ is some fixed measure on
$V'$. (Here we are choosing some measurable trivialization of the
bundle $E'$). Let $H \subset G$ denote the stabilizer of $\eta_0$.

If $H$ is a proper subgroup of $G$, \ann{\textbf{Referee:} Why can $H$ not be e.g. a proper parabolic subgroup in $G$?\\ \textbf{Authors:} We show that $H$ cannot be any proper subgroup of $G$. The proof works even if $H$ is a parabolic subgroup. In that case one can reach more specific conclusions: $H$ fixes some nontrivial line or subspace. And the measure will be supported there, so $G$ will act by a compact factor there.} we can reduce the algebraic hull over $\cM$ to $H$ as follows.
  First, by \cite[Corollary 3.2.23]{Zimmer_book} stabilizers of measures are algebraic subgroups, so there exists a tensor construction $ T(V') $ on the representation $ V' $, and a line $ l'\subset T(V') $ such that $ H $ is the stabilizer of $ l' $.
Moreover, for a.e. $ x\in \cM $ the linear map $\psi(x)^{-1}$
yields an isomorphism $ E_x' \to V' $ taking $ \eta_x$ to $ \eta_0$; such isomorphisms are parametrized by $ H $ acting by postcomposition on $ V' $.
  A linear isomorphism $ E_x'\to V' $ induces one on tensor constructions $ T(E_x')\to T(V') $ and so we can pull back the line $ l\subset T(V') $ to $ l_x\subset T(E_x') $.
  The linear isomorphism was well-defined up to postcomposition with $ H $, but $ l $ is $ H $-invariant so $ l_x $ is well-defined.
  The collection of lines $ l_x $ gives a further reduction of the algebraic hull over $ \cM $, which is not possible.
  Thus $G=H$, so $G$ leaves invariant a nontrivial measure $[\eta]$ on $\bP(V')$.

  % If the $G$-orbit of the measure is nontrivial, \ann{\textbf{Referee:} Why can $H$ not be e.g. a proper parabolic subgroup in $G$?\\ \textbf{Authors:} We show that $H$ cannot be any proper subgroup of $G$. The proof works even if $H$ is a parabolic subgroup. In that case one can reach more specific conclusions: $H$ fixes some nontrivial line or subspace. And the measure will be supported there, so $G$ will act by a compact factor there.} i.e. of the form $G/H$ with $H\neq G$, we can reduce the algebraic hull over $\cM$ to $H$ as follows.
  % First, by \cite[Corollary 3.2.23]{Zimmer_book} stabilizers of measures are algebraic subgroups, so there exists a tensor construction $ T(V') $ on the representation $ V' $, and a line $ l'\subset T(V') $ such that $ H $ is the stabilizer of $ l' $.
  % Moreover, there is a model measure $ \mu' $ on $ \bP(V') $ such that for a.e. $ x\in \cM $ there exists an isomorphism $ E_x' \to V' $ taking $ \mu_x'$ to $ \mu' $; such isomorphisms are parametrized by $ H $ acting by postcomposition on $ V' $.
  % A linear isomorphism $ E_x'\to V' $ induces one on tensor constructions $ T(E_x')\to T(V') $ and so we can pull back the line $ l\subset T(V') $ to $ l_x\subset T(E_x') $.
  % The linear isomorphism was well-defined up to postcomposition with $ H $, but $ l $ is $ H $-invariant so $ l_x $ is well-defined.
  % The collection of lines $ l_x $ gives a further reduction of the algebraic hull over $ \cM $, which is not possible.
  % Thus $G=H$, so $G$ leaves invariant a nontrivial measure $[\eta]$ on $\bP(V')$. 

  According to \cite[Corollary 3.2.2]{Zimmer_book} either the stabilizer of the measure $ [\eta] $ is compact, or there is a proper subspace of positive $ [\eta] $-mass which is left invariant by a finite index subgroup of $ G $. 
  But we assumed at the start that $G$ is connected (by passing to a finite cover) and $V'$ is $G$-irreducible, so it must be the case that $G$ acts on $V'$ via a compact group.
\end{proof}

\subsection{Second proof of Theorem \ref{thm:limit_absolute}}\label{SS:HT} We now give the second approach  to Theorem \ref{thm:limit_absolute}, which is related to \cite{MW}. In this section we assume all bundles are complexified.

Any tensor construction on $ H^1 $, denoted $ \mathbf{H} $, will admit a Hodge decomposition $ \mathbf{H} = \oplus \mathbf{H}^{p,q} $.
For establishing properties of the algebraic hull using \HT tensors, the following concepts will be useful.
Throughout, $ m\in \cM $ is some point, and $ \mathbf{H}_m $ denotes the fiber of $ \mathbf{H} $ at $ m $.

\begin{definition}[\HT tensor]
	\label{def:HTtensor}
	A pure Hodge-Teichm\"uller tensor at $m$ is an element of $\mathbf{H}_m$ of pure Hodge type $(p,q)$ for some $p,q$ such that the parallel transport along any path in the $\GL_2^+(\bR)$ orbit remains pure Hodge type $(p,q)$. A Hodge-Teichm\"uller tensor at $m$ is a linear combination of pure  Hodge-Teichm\"uller tensors. 
\end{definition}

\begin{prop}
	\label{prop:HT_properties}
	Let $ \mathbf{H} $ denote some fixed tensor construction on $ H^1 $.
	\leavevmode
	\begin{itemize}
		\item [(i)] The fiber of any $ \GL_2^+(\bR) $equivariant line subbundle of $\mathbf{H}$ is the span of a Hodge-Teichm\"uller tensor. 
		\item [(ii)] For any affine invariant submanifold $\cM$ there is a finite union $\cB$ of proper affine invariant submanifold of $\cM$ such that the Hodge-Teichm\"uller tensors in $\mathbf{H}$ form a continuous equivariant subbundle of $\mathbf{H}$ over $\cM\setminus \cB$. 
		\item [(iii)] The algebraic hull acts via isometries on any fiber of the bundle of Hodge-Teichm\"uller tensors.
	\end{itemize}
\end{prop}

\begin{proof}
	Part (i) is a direct consequence of \cite[Thm. 1.2]{sfilip_ssimple}.
	
	For (ii), the set of Hodge-Teichm\"uller tensors by definition is a $\GL_2^+(\bR)$ invariant subset of $\mathbf{H}$. Because the  Hodge decomposition and  $\GL_2^+(\bR)$-action are continuous, it is also a closed subset of $\mathbf{H}$. 
	
	Let $d$ be the minimal dimension of the space of Hodge-Teichm\"uller tensors for any point $m\in \cM$. Define $\cB$ to be the subset of $\cM$ where the dimension of the space of Hodge-Teichm\"uller tensors is strictly greater than $d$.
	Note that $ \cB $ is automatically a closed proper subset of $ \cM $.
	Over $\cM\setminus \cB$, the set of Hodge-Teichm\"uller tensors is vector subbundle of $\mathbf{H}$.
	
	For (iii) note that the bundle of Hodge-Teichm\"uller tensors of pure type $(p,q)$ forms an equivariant subbundle. The full bundle of Hodge-Teichm\"uller tensors is the direct sum of these  subbundles, so it suffices to prove the result for the bundle of pure $(p,q)$ Hodge-Teichm\"uller tensors. 
	
	The symplectic form on $H^1$ gives rise to a  locally constant bilinear form $B$ on $\mathbf{H}$. One can obtain a positive definite inner product from $B$ called the Hodge inner product by scaling $B$ by different signs on the different pieces of the Hodge decomposition. (This is part of the definition of a variation of Hodge structure.) 
	
	The tensor defining the bilinear form $B$ is flat and hence in particular equivariant, so by definition the algebraic hull preserves $B$. Hence the algebraic hull acts via isometries for the Hodge norm on a fiber of the  $(p,q)$ Hodge-Teichm\"uller tensors, since on this space the Hodge inner product and $B$ are proportional. 
\end{proof}

\begin{proof}[Second proof of Theorem~\ref{thm:limit_absolute}]
	The proof of the inclusion in part (i) is the same.
	
	For part (ii), suppose that $ \cM'\subset \cM $ is an affine manifold outside the locus $ \cB $ defined in Proposition~\ref{prop:HT_properties}.
	By definition of the algebraic hull and using Chevalley's Theorem~\ref{thm:chevalley}, the identity component of the algebraic hull of $ \cM' $ is the stabilizer of an equivariant polynomial line subbundle $\ell$ of the bundle of Hodge-Teichm\"uller tensors in some tensor construction $\mathbf{H}$. As in the first proof, only finitely many tensor constructions need to be considered.
	
	By part (iii) of Proposition~\ref{prop:HT_properties}, the algebraic hull of $ \cM $ acts by isometries on the bundle of \HT tensors.
	Therefore the kernel of the action on \HT tensors is cocompact in both connected components of the identity of the algebraic hulls of $ \cM $ and $ \cM' $, showing again that the hulls agree up to finite index and compact factors.
\end{proof}

\subsection{The relative case}
We will need several preliminaries before dealing with the algebraic hull in $ H^1_{rel} $.

\paragraph{Multiplicities}
The first step in controlling the algebraic hull of $ H^1_{rel} $ is to reduce to the case where bundles in $ H^1 $ have no multiplicity. (This is not necessary for subbundles of $H^1_{rel}$ whose image in $H^1$ have no multiplicity, such as $T(\cM)$. In particular, this is not required for our finiteness applications. The reader willing to ignore multiplicities may proceed directly to Theorem \ref{thm:limit_relative}.)

Let $ \cM $ be an affine invariant manifold and suppose that $ E\subset H^1 $ is a bundle with multiplicities, i.e. $ E = E_{irr}\otimes W $ where $ E_{irr} $ is an irreducible bundle and $ W $ is a vector space parametrizing the isotypical components.
Let $ \tilde{E}:=p^{-1}(E) $ denote the associated bundle in relative cohomology.
So we have a short exact sequence
\begin{align}
\label{eqn:ses_mult}
0 \to \ker p \to \tilde{E} \to E_{irr}\otimes W \to 0.
\end{align}
We would like to reduce to the case where there is no multiplicity in the pure weight $ 1 $ Hodge structure on the right.
For this, take a tensor with the dual $ W^\vee $ to obtain
\begin{align*}
0 \to \ker p\otimes W^\vee \to \tilde{E}\otimes W^\vee \xrightarrow{p} E_{irr}\otimes W\otimes W^\vee \to 0.
\end{align*}
Identifying $ W\otimes  W^\vee = \End(W) $ we have the direct sum decomposition $ W\otimes W^\vee = Id\oplus (\textrm{trace }0) $ where $ Id $ denotes multiples of the identity, and $ (\textrm{trace }0) $ denotes the trace $ 0 $ endomorphisms.
The factor $ E_{irr}\otimes Id = E_{irr} $ is present on the right-hand side above, so we take its preimage to obtain \begin{align}
\label{eqn:ses_no_mult}
0 \to \ker p\otimes W^\vee \to \widetilde{E_{irr}} \to E_{irr} \to 0
\end{align}
where $ \widetilde{E_{irr}} = p^{-1}(E_{irr}\otimes Id) $.
The advantage is that now $ E_{irr} $ is irreducible.

\begin{prop}
	\label{prop:mult}
	Suppose that for an affine manifold $ \cM'\subset \cM $ the algebraic hull of $ \widetilde{E_{irr}} $ from \eqref{eqn:ses_no_mult} over $ \cM' $ agrees up to finite index and compact factors with that over $ \cM $.
	Then the same holds for $ \tilde{E} $ from \eqref{eqn:ses_mult}.
\end{prop}
\begin{proof}
	It suffices to reconstruct the sequence \eqref{eqn:ses_mult} from \eqref{eqn:ses_no_mult} by natural operations.
	For this, take a tensor with $ W $ in \eqref{eqn:ses_no_mult} to obtain
	\begin{align*}
	0 \to \ker p\otimes W^\vee \otimes W \to \widetilde{E_{irr}}\otimes W \to E_{irr}\otimes W \to 0.
	\end{align*}
	
 Consider the commutative diagram below, which involves the natural quotient map $ q: \widetilde{E_{irr}}\onto E_{irr} $, the inclusion $ i:\widetilde{E_{irr}}\to \widetilde{E}\otimes W^\vee $, and for a bundle $ X $ the identity homomorphism $ \bold{1}_X $ (to be distinguished from $ id $ viewed as a vector in $ X\otimes X^\vee $).
\begin{center}
\resizebox{\textwidth}{!}{
	$$\xymatrix{
0 \ar[r] & \ker p\otimes W^\vee \otimes W \ar[r] \ar[d] & \widetilde{E_{irr}}\otimes W \ar[r]^{q\otimes \bold{1}_W} \ar[d]^{i\otimes \bold{1}_W} & E_{irr}\otimes W \ar[r]\ar[d]^{\bold{1}_{E_{irr}}\otimes id \otimes \bold{1}_W} & 0\\
0 \ar[r] & \ker p \otimes W^\vee \otimes W \ar[r] \ar[d] &  \widetilde{E}\otimes W^\vee \otimes W  \ar[r]_{p\otimes \bold{1}_{W^\vee\otimes W}} \ar[d] & E_{irr} \otimes W \otimes W^\vee \otimes W \ar[r] \ar[d] & 0\\
0 \ar[r] &\ker p \ar[r] &\widetilde{E} \ar[r]  & E_{irr}\otimes W  \ar[r] & 0\\
}$$ 
}
\end{center}

The map from the second row to the last is simply quotienting out by the trace $ 0 $ part of $ W^\vee \otimes W $. The commutativity of the upper-right corner of the diagram follows from the construction on $ \widetilde{E_{irr}} $ by tensoring with $ W $ the corresponding maps. The composition in the last column from top to bottom is an isomorphism (as can be checked by selecting a basis of $ W $) and the middle column is a surjection with kernel $ \ker p \otimes (\textrm{trace }0) $.

We thus obtain that  \eqref{eqn:ses_mult} can be obtain from \eqref{eqn:ses_no_mult} by first tensoring with $ W $and then quotienting the left and middle terms by $ \ker p \otimes (\textrm{trace }0) $.
\end{proof}

\begin{remark}In the exact sequence \eqref{eqn:ses_no_mult} the term $ \ker p \otimes W^\vee $ is still a trivial vector space, not a bundle.
So in all arguments below, $ \ker p $ can still be treated as a trivial bundle, even if we are in the case with multiplicities.
\end{remark}

To show containment of algebraic hulls in $ H^1_{rel} $ when $ \cM'\subset \cM $, we will need an analyticity result similar to the one for $H^1$.
It is established in the appendix, in Proposition~\ref{prop:an_H1rel}.

\begin{thm}
  \label{thm:limit_relative}
  Suppose $\cM_i$ is a sequence of affine invariant submanifolds of $\cM$ that equidistribute towards  $\cM$.
  
  Let $G_i$ and $G$ be the algebraic hulls of $\cM_i$ and $\cM$ respectively, for the $\GL_2^+(\bR)$-action on the relative cohomology bundle $H^1_{rel}$.
  \begin{itemize}
   \item [(i)] We have $G_i\subseteq G$ for all $i$.
   \item [(ii)] There exists $N\geq 0$ such that for all $i\geq N$, the groups $G_i$ and $G$ agree up to finite index and compact factors.
  \end{itemize}
\end{thm}

\begin{proof}
  Part (i) follows from Proposition \ref{prop:an_H1rel} since any reduction of the algebraic hull of $\cM$ is necessarily real-analytic, so it descends to $\cM_i\subset \cM$.
  
  To prove part (ii), note that by Theorem \ref{thm:limit_absolute},  in absolute cohomology  the algebraic hull stabilizes (up to compact factors and finite index) on $\cM_i$ for $i\gg 0$.
  
  Assume therefore that $E\subset H^1$ is an irreducible piece for the action of the reductive part of the algebraic hull; by Proposition~\ref{prop:mult} we can assume irreducibility of $ E $.
  It suffices to show that in the short exact sequence of bundles
  \begin{align}
    0 \to \ker(p) \to p^{-1}(E) \to E \to 0
  \end{align}
  the unipotent part of the algebraic hull stabilizes.
  By Lemma~\ref{lemma:S_alg_hull}(i)\&(ii) the unipotent part is a subbundle $S_i\subset \Hom(E,\ker p)$, which is invariant under the algebraic hull of $ E $.
  
    Since $ E $ is irreducible under the action of the algebraic hull,  $ S_i = \Hom(E,W_i) $ for some subspace $ W_i\subset \ker p $ .
Passing to a subsequence we obtain an accumulation point $W_\infty$ of $ W_i $. Define   $ S_\infty =\Hom(E,W_\infty)  \subset \Hom(E,\ker p) $.
    
  Note that the bundle of holomorphic $ 1 $-forms $ H^{1,0} \subset H^{1} $ lifts to a subbundle denoted $ \widetilde{H^{1,0}}\subset H^{1}_{rel} $, where lifting means that $ p: \widetilde{H^{1,0}} \to H^{1,0} $ is an isomorphism.

Lemma~\ref{lemma:S_alg_hull}(iii) shows that the $ S_i $ have the following description.
  Moving the bundle of $ 1 $-forms $ \widetilde{E^{1,0}}\oplus \widetilde{E^{0,1}} \subset H^{1}_{rel} $ by parallel transport along $ \GL_2^+(\bR) $-orbits in $ \cM_i $, it can be taken to its value at the new point by transformations in $ S_i $.
  
  It follows that $ S_\infty $ has the same property in $ \cM $, thus $ S_\infty $ contains the unipotent part of the algebraic hull of $ \cM $; they must agree since the $ S_i $ are contained in the algebraic hull of $ \cM $.
\end{proof}
\begin{remark}
	The above proof gives a $ \GL_2^+(\bR) $-invariant closed locus $ \cB_{rel}\subset \cM $ such that as soon as an affine submanifold $ \cM'\subset \cM $ is disjoint from $ \cB_{rel} $, the unipotent parts of the algebraic hulls agree.
	
	Indeed, above a fixed point $ x\in \cM $ we have a closed subset of the Grassmanian consisting of subspaces $ S_x\subset \Hom(E_x,\ker p) $ for which the defining property in Lemma~\ref{lemma:S_alg_hull}(iii) holds in a neighborhood of $ x $ inside its $ \GL_2^+(\bR) $-orbit.
	The subset of the total Grassmanian bundle is closed and $ \GL_2^+(\bR) $-invariant, and $ \cM $ is stratified by the minimal possible dimension of an $ S $.
	Away from a proper closed subset, this subspace is unique over the entire $ \cM $.
\end{remark}

%%%%%%%%%%%%%%%%%%%%%%%%%%%%%%%%%%%%%%%%%%%%
%%%%%%%%%%%%%%%%%%%%%%%%%%%%%%%%%%%%%%%%%%%%
\section{Finiteness and abundance results}\label{S:fin}
%%%%%%%%%%%%%%%%%%%%%%%%%%%%%%%%%%%%%%%%%%%%
%%%%%%%%%%%%%%%%%%%%%%%%%%%%%%%%%%%%%%%%%%%%

In this section we prove Theorems \ref{T:Fin} and \ref{T:H2}.
Recall that $ A_V(\cM) $ and $ G_V(\cM) $ denote the algebraic hull, and the Zariski closure of monodromy, on a flat bundle $ V $ in relative or absolute cohomology over $ \cM $.

%%%%%%%%%%%%%%%%%%%%%%%%%%%%%%%%%%%%%%%%%%%%
\subsection{Finiteness.}
%%%%%%%%%%%%%%%%%%%%%%%%%%%%%%%%%%%%%%%%%%%%
If $\cM'\subset \cM$, the restriction any flat bundle on $\cM$ gives a flat bundle on $\cM'$. 

\begin{lem}\label{L:finite}
	Suppose $\cM'\subsetneq \cM$ both have the same algebraic hull for $T\cM$ and its Galois conjugates. Then either
	\begin{itemize}
		\item $\cM'$ is rank 1 degree 1, or
		\item  $\cM'$ is rank 1 degree 2 and $\cM$ is rank 2 degree 1 and $\ker(p)\cap T\cM= \ker(p)\cap T\cM'$. 
	\end{itemize}
\end{lem}

\begin{proof}
	Theorem \ref{T:MoreAlgHull} implies that 
	\begin{enumerate}
		\item[(O1)] the only proper subspaces of $p(T\cM)$ invariant under $A_{p(T\cM)}(\cM)$ are the tautological plane and its complement, 
		\item[(O2)] the only proper subspace of $T\cM$ invariant under $A_{T\cM}(\cM)$ and not contained in  $p^{-1}$ of the tautological plane is $p^{-1}$ of the complement of the tautological plane in  $p(T\cM)$,
		\item[(O3)] there are no proper subspaces of a nontrivial Galois conjugate $p(T\cM)_\iota$ invariant under $A_{p(T\cM)_\iota}(\cM)$,
		\item[(O4)] there is no proper subspace of $(T\cM)_\iota$ that is not contained in $\ker(p)$ and is invariant under $A_{(T\cM)_\iota}(\cM)$.
	\end{enumerate}  
	
	Since the algebraic hulls of $\cM$ and $\cM'$ are the same, the algebraic hull of $\cM$ must preserve flat subbundles over $\cM'$. (Technically in such statements we should refer to fibers of bundles at points of $\cM'$, but throughout this proof we omit this specification.) In particular, $p(T\cM')$ must be stabilized by $A_{p(T\cM)}(\cM)$. By (O1), $p(T\cM')$ is either the tautological plane or $p(T\cM)$. 
	
	Suppose $p(T\cM')$ is not the tautological plane. So  $p(T\cM')=p(T\cM)$, and $\cM'$ has rank greater than 1. Since $T\cM'$ must be stabilized by $A_{T\cM}(\cM)$, (O2) implies $T\cM'= T\cM$. This contradicts $\cM'\neq \cM$. 
	
	Hence  $p(T\cM')$ is the tautological plane. So $\cM'$ is rank 1.  Note that any Galois conjugate of $p(T\cM')$ must be contained in some Galois conjugate of $p(T\cM)$. (This is a triviality about subspaces of vector spaces.) Any Galois conjugate of $p(T\cM')$ has dimension 2 and is stabilized by the algebraic hull of a Galois conjugate of $p(T\cM)$. 
	
	If $\cM$ has rank at least 3, (O1) and (O3) imply that $p(T\cM)$ and its Galois conjugates contain only one subspace invariant under algebraic hull of dimension 2, namely the tautological plane. Hence $\cM'$ has degree 1. 
	
	If $\cM$ has rank 2, (O1) and (O3) imply that $p(T\cM)$ and its Galois conjugates contain only two subspaces invariant under algebraic hull of dimension 2, namely the tautological plane and its complement in $p(T\cM)$. Hence $\cM'$ has degree 1 or 2. If $\cM'$ has degree 2, then the Galois conjugate of $p(T\cM')$ must be the complement of the tautological plane in $p(T\cM)$. Hence $p(T\cM)$ is the sum of $p(T\cM')$ and its Galois conjugate, and so $\cM$ has degree 1. The Galois conjugate of $T\cM'$ must be stabilized by  $A_{T(\cM)}(\cM)$, so (O2) implies that it is $p^{-1}$ of the complement of the tautological plane. In particular, $\ker(p)\cap T\cM =  \ker(p)\cap T\cM$.
	
	If $\cM$ has rank 1 and is not degree 1, (O4) implies that a Galois conjugate of $T(\cM')$ must be equal to a Galois conjugate of $T(\cM)$. This implies $\cM=\cM'$, a contradiction.  
\end{proof}

\begin{proof}[Proof of Theorem \ref{T:Fin}.]
	Let $\cM_1, \cM_2, \ldots $ be an infinite sequence of affine invariant submanifolds in some stratum that have fixed rank, degree, and dimension. The closure of their union is a finite union of affine invariant submanifolds, so passing to a subsequence we may assume that the $\cM_i$ are contained in and equidistribute to a single affine invariant submanifold $\cM$. By Theorem \ref{T:equidistribution}, removing finitely many of the $\cM_i$ if necessary, we may assume that the algebraic hull of $T(\cM)$ and its Galois conjugates on $\cM$ are equal to the algebraic hull of their restrictions to $\cM_i$. 
	
	Lemma \ref{L:finite} now gives the result.
\end{proof}

%%%%%%%%%%%%%%%%%%%%%%%%%%%%%%%%%%%%%%%%%%%%
\subsection{Abundance.}
%%%%%%%%%%%%%%%%%%%%%%%%%%%%%%%%%%%%%%%%%%%%
Let $\cO$ be an order in a real quadratic field, for example $\cO=\bZ[\sqrt{D}]$, where $D$ is a positive integer. Let $\cM$ be an affine invariant submanifold of rank 2 and degree 1. Say that $(X, \omega)$ is an eigenform for real multiplication by $\cO$ if there is an action of $\cO$ on $p(T\cM)_{(X,\omega)}$ by linear transformations that are self-adjoint with respect to the symplectic form, preserve the integer lattice, and act on $p(\omega)$ via scalars.  

\begin{lem}\label{L:RM}
	The locus of eigenforms in $\cM$ for  real multiplication by $\cO$ is a finite (possibly empty) union of codimension 2 degree 2 rank 1 affine invariant submanifolds. 
\end{lem}

The proof is omitted and is almost identical to that of \cite[Proposition 2.5]{Wsurvey},  \cite[Theorem 7.2]{Mc}, and \cite[Theorem 3.2]{Mc2}, see also \cite[Section 7]{MMW}. 

The following proof arose from a conversation with Ronen Mukamel. 

\begin{proof}[Proof of Theorem \ref{T:H2}.]
	Let $U$ be an open subset of $\cM$, and pick $m\in U$. Using period coordinates, we  identify $U$ with an open set in $(T\cM)_m$. 
	
	Pick any order $\cO$ in any real quadratic field $K$, and consider any action of $\cO$ of $p(T\cM)_m$ by self-adjoint transformations that preserve the integer lattice. Since the transformations in $\cO$ are self-adjoint, they are diagonalizable, and since $\cO$ is abelian they preserve each others eigenspaces. 
	
	Let $v$ be an eigenvector for the action of $\cO$. Since $\cM$ is degree 1, we may define $Sp(p(T\cM)_m, \bQ)$ to be the group of symplectic transformations of $p(T\cM)_m$ preserving the set of rational points. Note $Sp(p(T\cM)_m, \bQ)$ is isomorphic to $Sp(4, \bQ)$, which is dense in $Sp(4, \bR)$. Hence we can find $\gamma\in Sp(p(T\cM)_m, \bQ)$ so that $\gamma v\in p(U)$. 
	
	We can define a new action of $\cO$ on $p(T\cM)_m$ by conjugating the original action by $\gamma$. The resulting action is via self-adjoint transformations that act via a scalar on $v'$. There exists $N$ so that $\gamma$ is in the set $Sp(4, \frac1N \bZ)$. The restriction of the resulting action to $\cO'=N\cO$ preserves the integer lattice. 
	
	Hence if we pick $(X, \omega)\in U$  with $p(\omega)=v'$, we get that $(X,\omega)$ is an eigenform for real multiplication by $\cO'$. By Lemma \ref{L:RM}, this gives the result. 
\end{proof}

%====================================================
%====================================================
%		Appendix
%====================================================
%====================================================

\appendix
\section{Analyticity and polynomiality of measurable bundles in relative cohomology}
\label{app:an_poly}

Throughout this appendix, we work over a fixed affine invariant submanifold $\cM$.
Over $\cM$ we have the exact sequence of bundles
\begin{align}
  \label{eqn:app_ses_H1rel}
  0 \to \ker p \to H^1_{rel} \xrightarrow{p} H^1 \to 0
\end{align}
and we assume that $\ker p$ is trivialized (e.g. by passing to a finite cover).
The bundles above are real, but their complexifications contain holomorphic subbundles $ H^{1,0} $ and $ \widetilde{H^{1,0}} $ of holomorphic $ 1 $-forms, inducing variations of Hodge structures.

\subsection{Analyticity}

To handle the relative cohomology bundle, the first step is to show that any measurable $ GL_2^+(\bR) $-equivariant subbundle in $ H^1_{rel} $ must in fact be real-analytic.
This extends \cite[Thm. 7.7]{sfilip_ssimple} and was used in the proof of Theorem \ref{thm:limit_relative}. 

Let $ E\subset H^{1} $ be an irreducible bundle over a fixed affine manifold $ \cM $, and let $ \widetilde{E}:= p^{-1}(E) $.
Recall that we have the bundles of holomorphic $ 1 $-forms $ E^{1,0}\subset E_\bC $ and $ \widetilde{E^{1,0}}\subset \widetilde{E_{\bC}} $, and the forgetful map $p$ is an isomorphism from $ \widetilde{E^{1,0}}$ to $ E^{1,0}$.

Recall that $\xi \in \Hom(E,\ker p)$ defines a unipotent automorphism $ v\mapsto v + \xi(p(v)) $ of $\widetilde{E}$, and all automorphisms of $ \widetilde{E} $ which act as the identity on $ \ker p $ and $ E $ are of this form. Hence the unipotent part of the algebraic hull of $ \widetilde{E} $, denoted $ S $, is naturally contained in $ \Hom(E,\ker p) $. Moreover, since $ E $ carries a polarized weight one variation of Hodge structures, so does $ \Hom(E,\ker p )=E^\vee \otimes \ker p $, where $ \ker p $ is equipped with the trivial Hodge structure.

\begin{lem}
	\label{lemma:S_alg_hull}
	With notation as above, we have:
	  \begin{itemize}
	  	\item [(i)] The bundle $ S\subset \Hom(E,\ker p) $ respects the Hodge structure, i.e. has a Hodge decomposition compatible with that of $ \Hom(E,\ker p) $.
	  	Moreover $ S $ is invariant by the algebraic hull of $ E $.
	  	\item [(ii)] The bundle $ S $ can be alternatively described as the smallest bundle with the following property.
	  	Consider the bundle $ \widetilde{E^{1,0}_x} $ and its complex conjugate denoted $ \widetilde{E^{0,1}_x} $, where $ x\in \cM $.
	  	Then for any $ g\in \GL_2^+(\bR) $ small, using parallel transport for flat identifications, the bundle $ \widetilde{E^{1,0}_x} \oplus \widetilde{E^{0,1}_x}$ can be taken to $ \widetilde{E^{1,0}_{gx}}\oplus \widetilde{E^{0,1}_{gx}} $ using a transformation in $ S $.
	  \end{itemize}
\end{lem}

\begin{proof}
	For part (i), note that the algebraic hull is $ \GL_2^+(\bR) $-invariant (viewed as a group above each point) so in particular $ S $ is an invariant bundle.
	Therefore $ S $ must also respect the Hodge structure, by the semisimplicity results established in \cite{sfilip_ssimple}, which apply to any $ \GL_2^+(\bR) $-invariant subbundle of a weight 1 variation of Hodge structure over $\cM$.
	
	To establish part (ii) let $ S' $ denote the bundle described in it. First we show $S\subset S'$. Indeed, at each point we may pick a basis for $\tilde{E}$ which consists of a fixed basis for $\ker(p)\cap \tilde{E}$ together with any basis of $ \widetilde{E^{1,0}_{x}}\oplus \widetilde{E^{0,1}_{x}} $. In this basis we may consider the ``the unipotent part'' of an element in $ \End{\widetilde{E}} $ preserving $ \ker p $, and by definition the unipotent part of the cocycle is contained in $S'$.

	Suppose therefore that the algebraic hull could have been reduced to have unipotent part $ S\subset S' $.
	The reduction to algebraic hull with unipotent part $ S$ means the following.
	At every point $ x\in \cM $, we can pick a subspace $ \widetilde{E}_x' \subset \widetilde{E}$, projecting isomorphically to $ E $.
	Moreover, $ \widetilde{E}_x' $ is well-defined up to the action of $ S $ on $ \widetilde{E} $, and these choices and ambiguities are $ \GL_2^+(\bR) $-invariant.
	
	So locally on a $ \GL_2^+(\bR) $-orbit we have a map $ \sigma_{hull}:E \to \widetilde{E} $ giving a section of the projection (i.e. a map which, when composed with the projection, gives the identity).
	The section $ \sigma_{hull} $ is well-defined up to the action of $ S $, and can be viewed as a flat section of $ \Hom(E,\widetilde{E})/S $.
	
	We also have the canonical section $ \sigma_{hol}:E \to \widetilde{E} $, which is defined as the inverse of $p$ restricted to $ \widetilde{E^{1,0}} $ plus its complex conjugate.
	The difference $ \sigma_{hull}-\sigma_{hol} $ is an element of $ \Hom(E,\ker p) $, well-defined up to the action of $ S $, since $ \sigma_{hull} $ is.
	Now the image of $ \sigma_{hull}-\sigma_{hol} $ in $Q:= \Hom(E,\ker p)/S $ is well-defined, and moreover $ Q $ carries a weight $ 1 $ variation of Hodge structure, since $ S $ is compatible with the Hodge structure.
	
	To finish, note that \cite[Thm. 4.2]{sfilip_algebraicity} applies here (although stated only for certain parts of $ H^1 $ and $ H^1_{rel} $, the proof works in the present context). 
	It implies that $ \sigma_{hull}- \sigma_{hol} $ vanishes in $ Q $, in particular the subspaces we started with $ \widetilde{E}_x' $ can be taken to $ \sigma_{hol}(E) $ by elements of $ S $; therefore we can assume that they are, in fact, equal.
	By the definition of $ S' $, this implies that $ S'=S $.
\end{proof}

Using the above result as a preliminary step, we can now establish the analyticity of the algebraic hull in relative cohomology.
Note that although the algebraic hull of $H^1$ and $H^1_{rel}$ differ in a unipotent part only, the key difference is in the \emph{lift} of the semisimple part from $ H^1 $ to $ H^1_{rel} $.

\begin{prop}
  \label{prop:an_H1rel}
  Let $G$ be the (measurable) algebraic hull of the Kon\-tse\-vich--Zorich cocycle over an affine manifold $\cM$ for the $\GL_2^+(\bR)$-action on the relative cohomology bundle $H^1_{rel}$.
  Then the reduction to $G$ can be made real-analytic.
\end{prop}
\begin{proof}
  Decompose the algebraic hull $G$ according to the short exact sequence \ref{eqn:app_ses_H1rel} into a unipotent and reductive part:
  \begin{align}
    0 \to G^u \to G \to G^{ss} \to 0.
  \end{align}
  According to \cite{sfilip_ssimple}, and as discussed in the proof of Theorem \ref{thm:limit_absolute} the tensors defining $G^{ss}$ can be picked real-analytically.
  Note that $G^u$ is a subgroup of $\Hom(H^1,\ker(p))$ and hence decomposes according to the action of $G^{ss}$ on $H^1$.
  
  Let $E\subset H^1$ be an irreducible piece for the action of $G^{ss}$ (by Proposition~\ref{prop:mult} we can reduce to this case).
  It suffices to check that $G^{u}$ is defined by real-analytic bundles on the piece $\Hom(E,\ker(p))$.
  Indeed, $G$ is of the form $\prod_E G_E^{ss}\ltimes G^u_E$ where $G_E^{ss}$ is the algebraic hull of an irreducible piece $E \subset H^1$ and $G^u_E \subset \Hom(E,\ker(p))$.
  
  The algebraic hull $G_E$ on  $p^{-1}(E)$ sits in the exact sequence
  \begin{align}
    0 \to G^u_E \to G_E \to G^{ss}_E \to 0.
  \end{align}
  The unipotent part $ S_E:=G^{u}_E\subset \Hom(E,\ker p) $ respects the Hodge structure and is invariant under $ G^{ss}_E $, by Lemma~\ref{lemma:S_alg_hull}(ii).
  Finally, part (iii) of the same Lemma shows how to real-analytically reduce the algebraic hull to have unipotent part contained in $ S_E $.
  Indeed, the bundles $ \widetilde{E^{1,0}} $ and its complex-conjugate vary real-analytically, and give a real-analytic splitting $ G_E \simeq G^{ss}_E \ltimes S_E $, where $ G^{ss}_E $ is viewed as acting via the splitting of $ p^{-1}(E)\to E  $ coming from $ \widetilde{E^{1,0}} $.
\end{proof}

% % % % % % % % % % % % % % % % % % %
\subsection{Polynomiality}
% % % % % % % % % % % % % % % % % % %

We can now establish polynomiality of the algebraic hull, where polynomiality is understood in the following sense.
Consider $\GL_2^+(\bR)$-invariant bundles $\tilde{E}\subset H^1_{rel}$, or perhaps some tensor powers thereof.
Note that because $H^1$ is a symplectic bundle, and $\ker p$ is trivial, the bundle $H^1_{rel}$ is equipped with a natural volume form.
Any bundle $ \widetilde{E} $ obtained by tensor constructions from subbundles in $ H^1_{rel} $ will be also filtered by bundles with volume forms, and so we can assume $ \widetilde{E} $ has a volume form.
Therefore, its top exterior power $\Lambda^{(\dim \tilde{E})} \tilde{E}$ carries a canonical trivializing vector, which is $\GL_2^+(\bR)$-equivariant and denoted $ v_{\tilde{E}} $.
The coordinates of $v_{\tilde{E}}$ in $ \Lambda^{\dim \tilde{E}} (H^{1}_{rel})$, in local flat trivializations of the bundles, give functions on $\cM$ (these are just the Pl\"ucker coordinates on a Grassmanian).
We will show that these functions are polynomial, when viewed in period coordinates on $\cM$.
This is meant in the sense of \eqref{eqn:loc_form_eqvt_section}, i.e. as polynomials divided by the area function to some power.

\begin{prop}
\label{prop:poly_H1rel}
 Suppose $\tilde{E}\subset H^1_{rel}$ is a measurable $\GL_2^+(\bR)$-invariant subbundle of the relative cohomology bundle, or some tensor power of $H^1_{rel}$.
 Then $\tilde{E}$ is in fact polynomial.
\end{prop}
\begin{proof}
  By Prop.~\ref{prop:an_H1rel}, the bundle $\tilde{E}$ is at least real-analytic.
  Next, recall that polynomiality in $H^1_{rel}$ or its tensor powers is understood in terms of Pl\"ucker coordinates of the bundle.
  In other words, we have an invariant section of some tensor power of $H^1_{rel}$ given by the top exterior power $\Lambda^{\dim \tilde{E}}\tilde{E}$ (normalized by a fixed volume form). 
  
  But now, the same proof as in Prop.~7.5 and Thm.~7.7 of
  \cite{sfilip_ssimple} applies to give that the section must
  necessarily be polynomial in period coordinates.

  Let us recall a sketch of proof.
  Let $\phi$ be some real-analytic, $\GL_2^+\bR$-equivariant section of some tensor power of the Kontsevich--Zorich cocycle (on $H^1_{rel}$).
  The first part, following the proof of Prop.~7.5 in loc.cit., is to show that $\phi$ is polynomial on each stable, or each unstable, leaf.
  The joint polynomiality is then established as in Prop.~7.6 and Thm.~7.7, as those are simply statements about polynomials.
  
  To establish polynomiality on a stable leaf, in a local chart around a point $x\in \cM$ define $\tilde{\phi}(x,v):=\phi(x+v)$ where $v$ is a (sufficiently small) tangent vector in the unstable direction.
  We then have a Taylor expansion
  \begin{align*}
  \tilde{\phi}(x,v) = \sum_\alpha c_\alpha(x)v^\alpha
  \end{align*}
  where $\alpha$ denotes a multi-index and $c_\alpha(x)$ are vectors in the same bundle as $\phi$.
  From the equivariance under the Teichm\"uller geodesic flow, for any large time $t$ such that $g_t(x)$ returns to the same chart, we have:
  \begin{align*}
  \sum_\alpha c_\alpha(x) v^\alpha = \sum_\alpha g_{-t}c_\alpha(g_t x)(dg_t v)^\alpha
  \end{align*}
  where to obtain the right-hand side, we have pulled back by $g_{-t}$ the expansion near $g_t x$.
  Note that $dg_t$ denotes the cocycle on the tangent space, which by Forni's spectral gap result for the Lyapunov spectrum has a definite contraction, linear in time.
  Thus if $c_\alpha$ must vanish for $\alpha$ sufficiently large depending on the spectral gap and the Lyapunov exponents of the bundle in which $\phi$ lives.
\end{proof}

% % % % % % % % % % % % % % % % % % %
\section{Algebraic Hulls and Bundles}
\label{app:alg_hull}
% % % % % % % % % % % % % % % % % % %

\subsection{Algebraic Hulls following Zimmer}
\label{app:zimmer_alg_hull}

\bold{Setup.}
Recall that Zimmer \cite[Sec. 4.2]{Zimmer_book} works in the following setup.
We have a group action on a space $ A \curvearrowright X $.

\begin{definition}
	\label{def:Zimmer_cocycle}
	A cocycle for the action of $ A $ on $ X $ is a map $ \alpha: A \times X \to \GL_n(\bR) $ satisfying the cocycle relation
	\begin{align*}
	\alpha(a_1,a_2 x) \cdot \alpha(a_2,x)= \alpha(a_1 a_2, x).
	\end{align*}
	This induces an action of $ A $ on the trivial vector bundle $ X\times \bR^n $ by linear transformations on the fiber.
\end{definition}

Moreover, there is a notion of equivalence (or cohomology) of cocycles.
Namely, two cocycles $ \alpha $ and $ \beta $ as above are equivalent if there exists $ C:X\to \GL_n(\bR) $ such that $ C(ax)\circ \alpha(a,x) = \beta(a,x)\circ C(x)  $.

To have a more intrinsic view, one can work with general bundles $ V\to X $, with a lift of the action of $ A $ to $ V $ by linear transformations on the fiber.
Then a description as in Definition~\ref{def:Zimmer_cocycle} is obtained by trivializing the bundle so that $ V\simeq X\times \bR^n $.
Different trivializations give cohomologous cocycles.

Recall also Zimmer's definition of the algebraic hull.
\begin{definition}
	\label{def:alg_hull_zimmer}
	The Zimmer algebraic hull of a cocycle $ \alpha $ is the smallest algebraic group $ H\subset \GL_n(\bR) $ such that $ \alpha $ is cohomologous to a cocycle $ \beta $ taking values in $ H $.
\end{definition}
Note that with this definition, the algebraic hull is well-defined only up to conjugacy, and is not a priori clear why it is even well-defined.
Namely, one has to check that if the cocycle can be conjugated to take values in $ H_1 $ and $ H_2 $, then there is a conjugation with values in $ H_1 \cap g H_2 g^{-1} $, for some $ g\in \GL_n(\bR) $.

\begin{prop}
	The algebraic hull as in Definition~\ref{def:alg_hull_tensor} of the main text (call it the tensor algebraic hull), and Zimmer's Definition~\ref{def:alg_hull_zimmer} are equivalent.
\end{prop}
Recall that by Chevalley's Theorem~\ref{thm:chevalley}, to define an algebraic subgroup of $ \GL_n(\bR) $ is equivalent to specifying a line in some tensor construction on $ \bR^n $, with the group being the stabilizer of the line.
\begin{proof}
	First, we check that the tensor algebraic hull contains the Zimmer algebraic hull.
	Suppose that we have a cocycle $ V\to X $ and its algebraic hull is defined by some line subbundle $ l\subset \mathbf{V} $ in some tensor construction on $ V $.
	Fix a model line $ l_m \subset \mathbf{T}(\bR^n) $ in a corresponding tensor construction on $ \bR^n $.
	Then we pick a measurable trivialization of $ V $ such that under the identification of each fiber $ V_x \to \bR^n $, the lines $ l_x \subset \mathbf{V} $ are identified with $ l_m \subset \mathbf{T}(\bR^n) $.
	Thus we have reduced the cocycle in the sense of Zimmer to have algebraic hull contained in the tensor algebraic hull.
	
	Conversely, suppose that under some trivialization $ V\simeq X \times \bR^n $, all the cocycle linear transformations are in some group $ H\subset \GL_n(\bR) $.
	Then $ H $ is the stabilizer of some line $ l_m\subset  \mathbf{T}(\bR^n) $ in some tensor construction.
	The line $ l_m $ pulls back to give a line subbundle $ l\subset \mathbf{V} $ in the corresponding tensor construction on $ V $.
	By definition, the cocycle preserves the line bundle $ l $, so the tensor algebraic hull is contained in Zimmer's.
\end{proof}

\subsection{Irreducible and absolutely irreducible bundles}
To end, we clarify a point regarding absolute irreducibility of bundles.
It is not used in the main text, but shows that the analyticity results apply to both irreducible and strongly irreducible bundles.

Recall that a cocycle on a vector bundle $V$ is irreducible if it has no invariant subbundle.
Similarly, a cocycle is strongly irreducible if it does not admit a finite invariant collection of subbundles $W_i\subset V$.
For instance, monodromy could permute a finite number of subspaces.
Thus a cocycle can be irreducible without being strongly irreducible.

All the results above and in \cite{sfilip_ssimple} about analyticity and polynomiality of bundles refer to irreducible bundles.
They also apply to non-irreducible bundles, by simply decomposing them into irreducible pieces.
Below, we establish the same (local) property for a collection of invariant subspaces.

\begin{prop}
	Suppose that $V$ is some irreducible piece of a tensor power of $H^1$ over $\cM$, and suppose that $V$ is not strongly irreducible for the $\GL_2^+(\bR)$-action.
	By assumption, there exists a finite collection of bundles $W_i\subset V$ which are permuted by the $\GL_2^+(\bR)$-action.
	
	Then locally (after a renumbering), each of the bundles varies polynomially in period coordinates.
\end{prop}

Note that the numbering of the bundles can be pathological -- take any (measurable!) function from $\cM$ to permutations of the indices and relabel the bundles.
Part of the statement is that there is locally a relabeling for which the bundles vary polynomially.

\begin{proof}
	Consider the projectivizations $\bP(W_i)\subset \bP(V)$.
	Their union $X:=\cup \bP(W_i)$ is in each fiber a collection of linear spaces, i.e. an algebraic variety.
	Let $I^\bullet$ be (fiberwise) ideal of homogeneous polynomials which vanish on $X$.
	
	Note that each homogeneous component $I^k$ of $I^\bullet$ is a subbundle of some tensor power of the dual $V^\vee$.
	Therefore, each homogeneous component varies polynomially in period coordinates.
	
	Next, any sufficiently high homogeneous component of $I^\bullet$ determines the variety $X$.
	Therefore each of the bundles $W_i$ must (locally) vary polynomially in period coordinates, after an appropriate relabeling.
\end{proof}

%====================================================
%====================================================
%		End of Appendix
%====================================================
%====================================================

\bibliographystyle{amsalpha}
\bibliography{AlgHull}
\end{document}